\documentclass[twoside,openany,12pt]{book}

\usepackage[style=alphabetic, backend=bibtex, firstinits=true, doi=false, url=false, isbn=false, sorting=nyt]{biblatex}
\bibliography{Savage-Intro_Categorification-biblist.bib}
\AtEveryBibitem{
  \ifentrytype{book}{%
    \clearfield{pages}
  }{%
  }%
}
\renewbibmacro{in:}{
  \ifentrytype{article}{}{\printtext{\bibstring{in}\intitlepunct}}}
\DeclareFieldFormat[article,incollection,misc]{title}{#1}  
\DeclareFieldFormat{pages}{#1}  
\DeclareFieldFormat[article]{volume}{\mkbibbold{#1}}  
\renewbibmacro*{volume+number+eid}{%
  \printfield{volume}%
  \setunit*{\addcolon}
  \printfield{number}%
  \setunit{\addcomma\space}%
  \printfield{eid}}

\usepackage[colorlinks=true, pdfstartview=FitV, linkcolor=blue, citecolor=blue, urlcolor=blue, breaklinks=true]{hyperref}

\usepackage{makeidx,comment,etoolbox}

\usepackage{enumerate}
\usepackage{amsmath}	
\usepackage{amssymb}	
\usepackage{amsthm}
\usepackage{amsfonts}	
\usepackage{mathrsfs}	
\usepackage{paralist} 
\usepackage{tikz}
\usepackage{float}
\usepackage[all]{xy}

\renewcommand\emptyset\varnothing

\newcommand{\up}{\uparrow}
\newcommand{\down}{\downarrow}

\newcommand{\R}{\mathbb{R}}
\newcommand{\N}{\mathbb{N}}
\newcommand{\C}{\mathbb{C}}
\newcommand{\Q}{\mathbb{Q}}
\newcommand{\Z}{\mathbb{Z}}
\newcommand{\F}{\mathbb{F}}

\newcommand{\cA}{{\mathcal{A}}}

\newcommand{\cC}{{\mathcal{C}}}
\newcommand{\cD}{{\mathcal{D}}}

\newcommand{\cF}{{\mathcal{F}}}
\newcommand{\cG}{{\mathcal{G}}}
\newcommand{\cH}{{\mathcal{H}}}
\newcommand{\cK}{{\mathcal{K}}}
\newcommand{\cM}{{\mathcal{M}}}
\newcommand{\cN}{{\mathcal{N}}}

\newcommand{\cP}{{\mathcal{P}}}

\newcommand\fh{\mathfrak{h}}
\newcommand\fC{\mathfrak{C}}

\newcommand\id{\mathrm{id}}
\newcommand\md{\textup{-mod}}
\newcommand\Md{\textup{-Mod}}
\newcommand\pmd{\textup{-pmod}}
\newcommand\spl{\textup{split}}

\newcommand\pj{\textup{proj}}
\newcommand\Sy{\textup{Sym}}
\newcommand\Vect{\textup{Vect}}
\newcommand\Set{\textup{Set}}
\newcommand\bimod{\textup{bimod}}

\DeclareMathOperator{\im}{im}
\DeclareMathOperator{\End}{End}
\DeclareMathOperator{\Span}{Span}
\DeclareMathOperator{\Hom}{Hom}
\DeclareMathOperator{\Ind}{Ind}
\DeclareMathOperator{\Res}{Res}
\DeclareMathOperator{\Mor}{Mor}
\DeclareMathOperator{\Ob}{Ob}
\DeclareMathOperator{\Fun}{Fun}

\newcommand\fa{\text{for all }}   

\newtoggle{solutions}

\toggletrue{solutions}
\togglefalse{solutions}

\iftoggle{solutions}{%
  \newcommand{\solution}[1]{\textbf{Solution:} #1}
}{%
  \newcommand{\solution}[1]{}
}

\newtheoremstyle{custom-rem}
  {.7\baselineskip\@plus.2\baselineskip\@minus.2\baselineskip}
  {.7\baselineskip\@plus.2\baselineskip\@minus.2\baselineskip}
  {\upshape}
  {}
  {\itshape}
  {.}
  { }
  {}

\newtheoremstyle{custom-def}
  {.7\baselineskip\@plus.2\baselineskip\@minus.2\baselineskip}
  {.7\baselineskip\@plus.2\baselineskip\@minus.2\baselineskip}
  {\upshape}
  {}
  {\bfseries}
  {.}
  { }
  {}

\newtheoremstyle{custom-theo}
  {.7\baselineskip\@plus.2\baselineskip\@minus.2\baselineskip}
  {.7\baselineskip\@plus.2\baselineskip\@minus.2\baselineskip}
  {\itshape}
  {}
  {\bfseries}
  {.}
  { }
  {}

\theoremstyle{custom-theo}
\newtheorem{theorem}{Theorem}[section]
\newtheorem{lem}[theorem]{Lemma}
\newtheorem{prop}[theorem]{Proposition}
\newtheorem{cor}[theorem]{Corollary}

\theoremstyle{custom-rem}
\newtheorem{eg}[theorem]{Example}
\newtheorem{egs}[theorem]{Examples}
\newtheorem{rem}[theorem]{Remark}
\newtheorem{prob}{}[section]

\theoremstyle{custom-def}
\newtheorem{defin}[theorem]{Definition}

\newcommand{\define}[1]{\emph{#1}\index{#1}}
\newcommand{\Exercises}{\par \vspace{4ex plus 2ex minus 0ex}\hfil\rule{300pt}{0.5pt}\hfil\par
  \vspace{3ex plus 2ex minus 0ex}%
  \noindent\textbf{\Large{Exercises.}}\medskip}

\newcommand{\comments}[1]{ \begin{center} \parbox{5 in}{{\bf {\footnotesize Comments:  }}{\footnotesize \textit{#1}}} \end{center}}
\iftoggle{solutions}{}{\renewcommand{\comments}[1]{}}

\makeindex

\allowdisplaybreaks

\begin{document}
%


\thispagestyle{empty}

{\centering
\vspace*{0.08\textheight}
{\Huge\bfseries Introduction to Categorification}\\[\baselineskip]
{\scshape Winter School I}\\[\baselineskip]
{\scshape January 6--17, 2014}\\[\baselineskip]
{\scshape \href{http://www.crm.math.ca/LieTheory2014/}{New Directions in Lie Theory Thematic Semester}}\\[\baselineskip]
{\scshape \href{http://www.crm.umontreal.ca/en/}{Centre de Recherches Math\'ematiques}}\par
\vfill
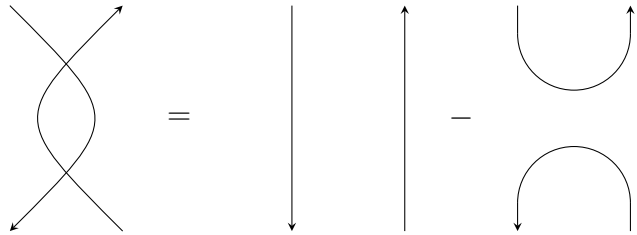
\begin{tikzpicture}[>=stealth,baseline=25pt,scale=1.5]
  \draw (0,0) .. controls (1,1) .. (0,2)[<-];
  \draw (1,0) .. controls (0,1) .. (1,2)[->] ;
  \draw (1.5,1) node {=};
  \draw (2.5,0) --(2.5,2)[<-];
  \draw (3.5,0) -- (3.5,2)[->];
  \draw (4,1) node {$-$};
  \draw (4.5,1.75) arc (180:360:.5) ;
  \draw (4.5,2) -- (4.5,1.75) ;
  \draw (5.5,2) -- (5.5,1.75) [<-];
  \draw (5.5,.25) arc (0:180:.5) ;
  \draw (5.5,0) -- (5.5,.25) ;
  \draw (4.5,0) -- (4.5,.25) [<-];
\end{tikzpicture}
\vfill
{\large\scshape \href{http://mysite.science.uottawa.ca/asavag2/index.html}{Alistair Savage}}\\[\baselineskip]
{\scshape Department of Mathematics and Statistics}\\[\baselineskip]
{\scshape University of Ottawa}\par
\vfill
This work is licensed under a \\ \href{http://creativecommons.org/licenses/by-sa/4.0/deed.en_US"}{Creative Commons Attribution-ShareAlike 4.0 International License}\par
\vspace*{0.08\textheight}}

\newpage

\pagenumbering{roman}

\pagestyle{headings}

\tableofcontents



\chapter*{Preface}
\addcontentsline{toc}{chapter}{Preface}
\pagestyle{plain}

These are the notes for a two-week mini-course \emph{Introduction to Categorification} given at a winter school in January 2014 as part of the thematic semester \emph{New Directions in Lie Theory} at the Centre de Recherches Math\'ematiques in Montr\'eal.  Each week, the course met four times for one-and-a-half hours.  There was also a problem session each week where students presented solutions to the exercises in the notes.

The goal of the course was to give an overview of the idea of categorification, with an emphasis on some examples where explicit computation is possible with minimal background.  Students were expected to have a solid understanding of groups, rings, and modules covered, for instance, in typical first year graduate courses in algebra.  A basic knowledge of the fundamentals of category theory was also assumed.

The course began with a very brief review of the representation theory of associative algebras, before introducing the concept of weak categorification with some simple examples.  It then proceeded to a discussion of more sophisticated examples of categorification, including a weak categorification of the polynomial representation of the Weyl group and the Fock space representation of the Heisenberg algebra.  The course concluded with a discussion of strong categorification and a brief overview of some further directions in the field.

Thank you to all of the students of the mini-course for their enthusiasm and interesting questions and comments.  Special thanks go to Yvan Saint-Aubin and Franco Saliola for useful remarks and for pointing out various typographical errors in earlier versions of these notes.

\bigskip
\noindent \href{http://mysite.science.uottawa.ca/asavag2/index.html}{Alistair Savage} \hfill Ottawa, 2014.

\bigskip \bigskip \bigskip \bigskip

\noindent \emph{Course website:} \url{http://mysite.science.uottawa.ca/asavag2/categorification/}

\pagestyle{headings}

%
\chapter{A brief review of modules over associative algebras}
\chaptermark{Associative algebras}
\pagestyle{headings}
\pagenumbering{arabic}
%

While categorification can take place in a very general setting, we will be concerned in this course almost exclusively with categories of modules over rings or algebras.  Thus, in order to speed up the exposition and get more quickly to some interesting results in categorification, we will work with such categories from the start.  In this chapter, we will review some of the key properties of these categories that will be used in the sequel.

%
\section{Associative algebras and their modules} \label{sec:assoc-alg}
%

Throughout this chapter, we fix an arbitrary field $\F$.  We first recall the definition of an associative algebra.

\begin{defin}[Associative algebra]
  Suppose $R$ is a commutative ring.  An \emph{associative $R$-algebra}\index{associative algebra}\index{algebra!associative} is a ring $B$ that is also an $R$-module and such that the ring multiplication is $R$-bilinear:
  \[
    \alpha (ab) = (\alpha a)b = a(\alpha b),\quad \fa \alpha \in R,\ a,b \in B.
  \]
  We say that $B$ is \emph{unital}\index{unital algebra}\index{algebra!unital} if it contains an element $1$ (denoted $1_B$ when there is chance of confusion) such that
  \[
    1 b = b = b 1,\quad \fa b \in B.
  \]
  If $B$ is commtutative (as a ring), then we say that it is a \emph{commutative $R$-algebra}\index{commutative algebra}\index{algebra!commutative}.
\end{defin}

\begin{eg}
  A $\Z$-algebra is the same as a ring.
\end{eg}

In these notes, we will be most interested in the case of algebras over a field (i.e.\ we will take $R=\F$).

\begin{eg}[Group algebra]
  If $\Gamma$ is a group, then we can define the \define{group algebra} $\F [\Gamma]$.  As an $\F$-module, $\F [\Gamma]$ is the $\F$-vector space with basis $\Gamma$.  Multiplication is given by
  \[
    (\alpha_1 \gamma_1)(\alpha_2 \gamma_2) = (\alpha_1 \alpha_2) (\gamma_1 \gamma_2),\quad \fa \alpha_1, \alpha_2 \in \F,\ \gamma_1, \gamma_2 \in \Gamma,
  \]
  and extending by linearity.
\end{eg}

\begin{eg}[Endomorphism algebra]
  If $V$ is an $\F$-vector space, then $\End V$ is an algebra under the natural operations (e.g.\ multiplication is given by composition of endomorphisms).  It is called the \define{endomorphism algebra} of $V$.
\end{eg}

\begin{eg}[Algebra of dual numbers]
  The algebra $D = \F[x]/(x^2)$ is called the \define{algebra of dual numbers}\index{dual numbers}.  It is a two-dimensional algebra.
\end{eg}

\begin{defin}[Algebra homomorphism]
  An \define{algebra homomorphism} between two associative $R$-algebras is an $R$-linear ring homomorphism.  For a homomorphism $\psi$ of \emph{unital} associative $R$-algebras, we also require that $\psi(1)=1$.
\end{defin}

For the remainder of this section, we fix a unital associative $\F$-algebra $B$.

\begin{defin}[Algebra representation]
  A \define{representation} of $B$ is a unital algebra homomorphism $B \to \End V$ for some $\F$-module $V$.
\end{defin}

\begin{defin}[Module, simple module] \index{module}
  A left (resp.\ right) $B$-module is simply a left (resp.\ right) $B$-module for the underlying ring of $B$.  It follows that a left $B$-module $M$ is also an $\F$-module, with action
  \[
    \alpha m = (\alpha 1_B)m,\quad \fa \alpha \in \F,\ m \in M
  \]
  (and similarly for right $B$-modules).  A module is \emph{simple}\index{simple module}\index{module!simple} if it has no nonzero proper submodules.
\end{defin}

Suppose $M$ is a simple $B$-module and let $m$ be an arbitrary nonzero element of $M$.  Then $Bm = \{bm\ |\ b \in B\}$ is a nonzero submodule of $M$.  Since $M$ is simple, we must have $Bm=B$.  Thus, $M$ is generated by any nonzero element.  Now consider the homomorphism of $B$-modules
\[
  f \colon B \to M,\ b \mapsto bm.
\]
Then $f$ is surjective and so, by the First Isomorphism Theorem, we have $M \cong B/\ker f$ as $B$-modules. Since $M$ is simple, $\ker f$ is a maximal ideal of $B$.  In this way, we see that all simple $B$-modules are isomorphic to quotients of $B$ by maximal ideals.

\begin{egs} \label{egs:simple-module}
  \begin{enumerate}[(a)]
    \item The only simple $\F$-module (up to isomorphism) is the one-dimensional vector space $\F$.
    \item \label{eg-item:simple-module-dual-numbers} The only maximal ideal of the algebra $D$ of dual numbers\index{algebra of dual numbers}\index{dual numbers} is $(x)$ (see Exercise~\ref{prob:max-ideal-dual-numbers}).  Thus, the only simple $D$-module (up to isomorphism) is $D/(x)$.  This is a one-dimensional module on which $x$ acts by zero (and $1_D$ acts by the identity).
  \end{enumerate}
\end{egs}

The notions of representations of $B$ and left $B$-modules are equivalent.  Thus, by abuse of terminology, we will sometimes use the terms interchangeably.  Furthermore, when we write `$B$-module', without specifying `left' or `right', we shall mean `left $B$-module'.  Recall that a left $B$-module is \emph{finitely generated}\index{finitely generated module}\index{module!finitely generated} if it has a finite generating set.

A \define{short exact sequence} of modules is a sequence of $B$-module homomorphisms
\[
  0 \xrightarrow{\psi_0} M_1 \xrightarrow{\psi_1} M_2 \xrightarrow{\psi_2} M_3 \xrightarrow{\psi_3} 0,
\]
such that $\im \psi_i = \ker \psi_{i+1}$ for $i=0,1,2$.  We say that the sequence \emph{splits} (and it is a \define{split exact sequence}) if either one of the following two equivalent conditions is satisfied:
\begin{enumerate}[(a)]
  \item there exists a $B$-module homomorphism $\varphi_1 \colon M_2 \to M_1$ such that $\varphi_1 \psi_1 = \id_{M_1}$;
  \item there exists a $B$-module homomorphism $\varphi_2 \colon M_3 \to M_2$ such that $\psi_2 \varphi_2 = \id_{M_3}$.
\end{enumerate}
In this case we have that $M_2 \cong M_1 \oplus M_3$.

\begin{defin}[Projective module] \label{def:projective}
  A $B$-module $P$ is \emph{projective}\index{projective module}\index{module!projective} if every short exact sequence of the form
  \[
    0 \to M \to N \to P \to 0
  \]
  (where $M$ and $N$ are $B$-modules) is a split exact sequence.
\end{defin}

A $B$-module $P$ is projective if and only if it is a direct summand of a free module (in particular, free modules are projective).  Another characterization of projective modules is as follows.  Remember that the functor $\Hom (M, -)$ from the category of $B$-modules to the category of abelian groups is always left exact (for any $B$-module $M$).  It is also right exact (hence exact) if and only if $M$ is projective.

\begin{defin}[Superfluous submodule, superfluous epimorphism]
  A submodule $N$ of a $B$-module $M$ is \emph{superfluous}\index{superfluous submodule}\index{submodule!superfluous} if, for any other submodule $H$ of $M$, the equality $N + H = M$ implies $H=M$.  A \define{superfluous epimorphism}\index{epimorphism!superfluous} of $B$-modules is an epimorphism $p \colon M \to N$ whose kernel is a superfluous submodule of $M$.
\end{defin}

\begin{eg}
  The zero submodule is always superfluous.  A nonzero $B$-module $M$ is never superfluous in itself (take $H=0$ in the definition of a superfluous module).
\end{eg}

\begin{defin}[Projective cover]
  Suppose $M$ is a $B$-module.  A \define{projective cover} of $M$ is a projective module $P$, together with a superfluous epimorphism $P \to M$ of $B$-modules.
\end{defin}

\begin{egs} \label{egs:proj-cover}
  \begin{enumerate}[(a)]
    \item Any projective module is its own projective cover.  For example, the projective cover of the $\F$-module $\F$ is $\F$.
    \item \label{eg-item:proj-cover-dual-numbers} The projective cover of the simple $D$-module $D/(x)$ is $D$ itself (see Exercise~\ref{prob:proj-cover-dual-numbers}).\index{algebra of dual numbers}\index{dual numbers}
  \end{enumerate}
\end{egs}

When they exist, the projective cover and associated superfluous epimorphism of a given module $M$ is unique up to isomorphism.  However, in general, projective covers (of modules of arbitrary rings/algebras) need not exist.

There is no natural definition of the (internal) tensor product of two $B$-modules unless one has some additional structure on $B$ (e.g.\ the structure of a Hopf algebra).  However, we will often use the notion of an \define{external tensor product}.  If $A$ and $B$ are both unital associative algebras, $M$ is an $A$-module, and $N$ is a $B$-module, then $M \otimes_\F N$ is an $(A \otimes_\F B)$-module via the action
\[
  (a \otimes b)(m \otimes n) = (am) \otimes (bn),\quad \fa a \in A,\ b \in B,\ m \in M,\ n \in N
\]
(and extending by linearity).

\Exercises

\medskip

\begin{prob} \label{prob:max-ideal-dual-numbers}
  Show that the only maximal ideal of the algebra $D$ of dual numbers is $(x)$.\index{algebra of dual numbers}\index{dual numbers}
\end{prob}

\begin{prob} \label{prob:proj-cover-dual-numbers}
  Show that $D$ is the projective cover of the $D$-module $D/(x)$.\index{algebra of dual numbers}\index{dual numbers}
\end{prob}

%
\section{Finite-dimensional algebras}
%

In these notes, we will be most concerned with finite-dimensional unital associative $\F$-algebras (i.e.\ unital associative $\F$-algebras that are finite-dimensional as an $\F$-vector space).  In this section, we will state some of the important properties of such algebras.  Many of these properties follow from the fact that finite-dimensional algebras are \emph{Artinian rings}\index{Artinian ring}\index{ring!Artinian} (see Exercise~\ref{prob:fd-implies-Artinian}).  A good reference for the properties of Artinian rings and their modules is~\cite{ARS95}.

Throughout this section, we assume that $B$ is a finite-dimensional unital associative $\F$-algebra.  We also assume that all modules are finitely generated.

\begin{prop}[Properties of modules over finite-dimensional algebras] \label{prop:fd-alg-modules-props}
  Recall that $B$ is a finite-dimensional unital associative algebra over a field $\F$.
  \begin{enumerate}[(a)]
    \item Every left (resp.\ right) $B$-module has a projective cover.  (See~\cite[Th.~I.4.2]{ARS95}.)
    \item Every flat (left or right) $B$-module is projective.  (Recall that a right $B$-module $M$ is \emph{flat}\index{flat module}\index{module!flat} if the functor $M \otimes_B -$ maps exact sequences of left $B$-modules to exact sequences of abelian groups (and similarly with left and right interchanged).  Since projective modules are always flat, we have that the notions of flat and projective $B$-modules are equivalent.)  (See~\cite[Th.~28.4]{AF92}.)
    \item The algebra $B$ has a finite number of nonisomorphic simple modules (see \cite[Prop.~I.3.1]{ARS95}).
    \item The projective covers of the (nonisomorphic) simple modules form a complete list of nonisomorphic indecomposable projective $B$-modules.  (See~\cite[Cor.~I.4.5]{ARS95}.)
  \end{enumerate}
\end{prop}

\begin{eg} \label{eg:dual-numbers-simples-and-projs}\index{algebra of dual numbers}\index{dual numbers}
  As seen in Examples~\ref{egs:simple-module}\eqref{eg-item:simple-module-dual-numbers} and~\ref{egs:proj-cover}\eqref{eg-item:proj-cover-dual-numbers}, the algebra $D$ of dual numbers has one simple module, namely $D/(x)$.  Its projective cover $D$ is the only indecomposable projective $D$-module (up to isomorphism).
\end{eg}

The proof of the following lemma is an exercise.

\begin{lem} \label{lem:hom-proj-simple}
  Suppose $V$ is a simple $B$-module with projective cover $P$.  Then, for any simple $B$-module $W$, we have an isomorphism of $\F$-modules
  \[
    \Hom_B (P,W) \cong \Hom_B(V,W) =
    \begin{cases}
      0 & \text{if } W \not \cong V, \\
      \End_B(V) & \text{if } W \cong V,
    \end{cases}
  \]
  where the conditions $W \not \cong V$ and $W \cong V$ refer to isomorphisms of $B$-modules.
\end{lem}

\begin{defin}[Simple algebra]
  An associative algebra $B$ is called \emph{simple}\index{simple algebra}\index{algebra!simple} if it has no nontrivial proper (two-sided) ideals and $B^2 = \{ab\ |\ a,b \in B\} \ne \{0\}$.  (Note that if $B$ is a nonzero unital associative algebra, then the second condition is automatically satisfied.)
\end{defin}

\begin{eg}
  The algebra of $n \times n$ matrices ($n \ge 1$) with entries in $\F$ is a simple $\F$-algebra.
\end{eg}

\begin{defin}[Semisimple]
  A finite-dimensional unital associative algebra is \emph{semisimple}\index{semisimple algebra}\index{algebra!semisimple} if it is isomorphic to a Cartesian product of simple subalgebras. A module over an associative algebra is \emph{semisimple}\index{semisimple module}\index{module!semisimple} if it is isomorphic to a direct sum of simple submodules.
\end{defin}

\begin{lem}[Maschke's Theorem] \index{Maschke's Theorem} \label{lem:Maschke}
  Suppose $\Gamma$ is a finite group and the characteristic of the field $\F$ does not divide the order of $\Gamma$.  Then the group algebra $\F[\Gamma]$ is semisimple.
\end{lem}

\begin{prop}[Properties of modules over semisimple algebras] \label{prop:semisimple-module-props}
  Suppose $B$ is a semisimple finite-dimensional unital associative algebra.
  \begin{enumerate}[(a)]
    \item All $B$-modules are semisimple.
    \item All $B$-modules are projective.  In particular, every $B$-module is its own projective cover.
  \end{enumerate}
\end{prop}

\Exercises

\medskip

\begin{prob} \label{prob:fd-implies-Artinian}
  Prove that a finite-dimensional unital associative algebra is an Artinian ring.  (Recall that a ring is \emph{Artinian}\index{Artinian ring}\index{ring!Artinian} if it satisfies the descending chain condition on ideals.)
\end{prob}

\begin{prob}
  Prove Lemma~\ref{lem:hom-proj-simple}.
\end{prob}

\solution{
  By Schur's Lemma, we have $\Hom_B(V,W)=0$ if $V$ and $W$ are nonisomorphic simple $B$-modules.  Thus, it suffices to prove the isomorphism $\Hom_B(P,W) \cong \Hom_B(V,W)$.  Let $p \colon P \to V$ be a superfluous epimorphism.  We will show that the map
  \[
    \varphi \colon \Hom_B(V,W) \to \Hom_B(P,W),\quad \varphi(f) = f \circ p,
  \]
  is an isomorphism.  We first show that $\varphi$ is surjective.  Let $g \in \Hom_B(P,W)$.  Since $V$ and $W$ are simple, the ideals $\ker p$ and $\ker g$ are maximal.  If $W \not \cong V$, then $\ker p \ne \ker g$ and so we must have $\ker p + \ker g = P$.  Since $\ker p$ is superfluous, this implies that $\ker g = P$, so $g=0$.  Hence $\Hom_B(P,W)=0$ and $\varphi$ is surjective.  If $f \ne 0$, then we must have $\ker p = \ker f$.  Thus $f$ factors through $P/(\ker p) \cong V$.  In other words, $f$ in in the image of $\varphi$.  So $\varphi$ is again surjective.  It remains to show that $\varphi$ is injective.  But this follows from the fact that, since $p$ is surjective, we have $g \circ p = g \circ p \implies f=g$ for all $f,g \in \Hom_B(V,W)$.
}

%
\chapter{Weak categorification}
\pagestyle{headings}
%

In this chapter we explain the idea of \define{weak categorification}\index{categorification}.  We begin by presenting an overview of the important ingredients: categories, functors, and Grothendieck groups.  We refer the reader to~\cite{LM12} for a more detailed introduction to Grothendieck groups (in the context of categorification) and operations induced on them by various additional structure on the category in question.

%
\section{Grothendieck groups}
%

We again fix a unital associative $\F$-algebra $B$.  Associated to $B$ are two categories that will be of particular interest to us.

\begin{defin}[Categories $B\md$ and $B\pmd$]
  We let $B\md$\index{mod} denote the category of finitely generated left $B$-modules and let $B\pmd$\index{pmod} denote the category of finitely generated projective left $B$-modules.  Thus $B\pmd$ is a full subcategory of $B\md$.
\end{defin}

Let $\cC$ be a subcategory of the category of $B$-modules.  We will be primarily interested in the case where $\cC$ is either $B\md$ or $B\pmd$.

\begin{defin}[Split Grothendieck group] \label{def:split-Groth}
  Let $F(\cC)$ be the free abelian group with basis the isomorphism classes $[M]$ of objects $M$ in $\cC$, and let $N^\spl(\cC)$ be the subgroup generated by the elements $[M_1] - [M_2] + [M_3]$ for every split exact sequence $0 \to M_1 \to M_2 \to M_3 \to 0$ in $\cC$ (equivalently, by $[M_3] - [M_1] - [M_2]$ for every $M_1, M_2, M_3 \in \cC$ with $M_3 = M_1 \oplus M_2$).  The \define{split Grothendieck group} of $\cC$, denoted $\cK^\spl_0(\cC)$, is the quotient group $F(\cC)/N^\spl(\cC)$.   We will usually denote the image of $[M]$ in $\cK^\spl_0(\cC)$ again by $[M]$.
\end{defin}

\begin{eg} \label{eg:Fmd-split-Groth-group}\index{K@$\cK^\spl_0$}
  Note that $\F\md$ is the category of all finite-dimensional $\F$-vector spaces and $\cK^\spl_0(\F\md) \cong \Z$.  Indeed, consider the surjective homomorphism
  \[
    f \colon F(\F\md) \to \Z,\ f([V])=\dim(V)
  \]
  (and extended by linearity).  Since dimension is additive (i.e.\ $\dim(V \oplus W)=\dim(V)+\dim(W)$), we have $N^\spl(\F\md) \subseteq \ker(f)$.  Now, let $\sum_{i=1}^n c_i [V_i]$ be an arbitrary element of $\ker(f)$.  We have $\sum_{i=1}^n c_i \dim(V_i)= f(\sum_{i=1}^n c_i [V_i]) = 0$.  In $\cK^\spl_0(\F\md)$, since $[V_i]=\dim(V_i)[\F]$, we have $\sum_{i=1}^n c_i [V_i] = (\sum_{i=1}^n c_i \dim(V_i)) [\F] = 0$, so $\ker(f) = N^\spl(\F\md)$.  Thus, by the First Isomorphism Theorem,
  \[
    \cK^\spl_0(\F\md) \cong F(\F\md)/\ker(f) \cong \Z.
  \]
\end{eg}

\begin{defin}[Grothendieck group] \index{K@$\cK_0$}
  As in Definition~\ref{def:split-Groth}, let $F(\cC)$ be the free abelian group with basis the isomorphism classes $[M]$ of objects $M$ in $\cC$.  Let $N(\cC)$ be the subgroup of $F(\cC)$ generated by the elements $[M_1] - [M_2] + [M_3]$ for every short exact sequence $0 \to M_1 \to M_2 \to M_3 \to 0$ in $\cC$.  The \define{Grothendieck group} of $\cC$, denoted $\cK_0(\cC)$, is the quotient group $F(\cC)/N(\cC)$.   We will usually denote the image of $[M]$ in $\cK_0(\cC)$ again by $[M]$.
\end{defin}

\begin{rem}
  The definition of a Grothendieck group remains valid for any abelian category $\cC$.
\end{rem}

\begin{eg}
  Note that every short exact sequence in $\F\md$ splits.  Thus, the split Grothendieck group and the Grothendieck group of $\F\md$ are the same.  Therefore, by Example~\ref{eg:Fmd-split-Groth-group}, $\cK_0(\F\md) \cong \Z$.
\end{eg}

\begin{defin}[$G_0(B)$ and $K_0(B)$] \index{G@$G_0$}\index{K@$K_0$}
  Let $G_0(B) = \cK_0(B\md)$ and $K_0(B) = \cK_0(B\pmd)$.
\end{defin}

\begin{rem}
  By the definition of a projective module (Definition~\ref{def:projective}), every short exact sequence in $B\pmd$ splits.  Thus $K_0(B) = \cK_0^\spl(B\pmd)$.
\end{rem}

\begin{lem}
  If $B$ is semisimple, then all short exact sequences in $B$-mod split.  Thus all modules are projective.  Hence, $B\md=B\pmd$ and $G_0(B)=K_0(B)$.
\end{lem}

\begin{proof}
  This follows immediately from Proposition~\ref{prop:semisimple-module-props}.
\end{proof}

For the remainder of this section, we assume that $B$ is finite-dimensional.  Let $V_1,\dotsc,V_s$ be a complete list of nonisomorphic simple $B$-modules.  If $P_i$ is the projective cover of $V_i$ for $i=1,\dotsc,s$, then $P_1\dotsc,P_s$ is a complete list of nonisomorphic indecomposable projective $B$-modules (see Proposition~\ref{prop:fd-alg-modules-props}).  It follows from the Jordan-H\"older Theorem (see~\cite[Th.~I.1.7]{ARS95}) that
\[
  G_0(B) = \bigoplus_{i=1}^s \Z[V_i].
\]
The class $[M] \in G_0(B)$ of any $M \in B\md$ is the sum (with multiplicity) of the classes of the simple modules appearing in any composition series of $M$.

Since any $P \in B\pmd$ can be written uniquely as a sum of indecomposable projective modules, we also have
\[
  K_0(B) = \bigoplus_{i=1}^s \Z[P_i].
\]

\begin{eg} \label{eg:G0-K0-dual-numbers}\index{algebra of dual numbers}\index{dual numbers}
  By Example~\ref{eg:dual-numbers-simples-and-projs}, for the algebra $D$ of dual numbers we have
  \[
    G_0(D) = \Z[D/(x)],\quad K_0(D) = \Z[D].
  \]
\end{eg}

\begin{rem}
  Note that, in situations such as Example~\ref{eg:G0-K0-dual-numbers}, $\Z[D]$ denotes the one-dimensional $\Z$-module spanned by $[D]$.  This should not be confused with the notation $\Z[x]$, where $x$ is an indeterminate, which is the polynomial algebra in one variable.
\end{rem}

We have a natural bilinear form
\begin{gather}
  \langle -, - \rangle \colon K_0(B) \otimes_\Z G_0(B) \to \Z,\quad \text{given by} \nonumber \\
  \langle [P], [M] \rangle = \dim_\F \Hom_B(P,M),\quad \fa P \in B\pmd,\ M \in B\md \label{eq:bilinear-form}
\end{gather}
(extending by bilinearity).  Here $\Hom_B(P,M)$ denotes the $\F$-vector space of all $B$-module homomorphisms from $P$ to $M$.  This form is well-defined in the first argument since $\Hom_B(-,M)$ is an additive functor for all $M \in B\md$, and $K_0(B) = \cK^\spl_0(B\pmd)$.  It is well-defined in the second argument since $\Hom_B(P,-)$ is an exact functor for $P \in B\pmd$.  Note that it is crucial that $P$ is projective here.  For example, we do not, in general, have an analogous bilinear form $G_0(B) \otimes G_0(B) \to \Z$ (unless, of course, $B$ is semisimple, in which case $G_0(B)=K_0(B)$).

By Lemma~\ref{lem:hom-proj-simple}, we have
\begin{equation} \label{eq:pairing-orthogonal}
  \langle [P_i], [V_j] \rangle =
  \begin{cases}
    0 & \text{if } i \ne j, \\
    \dim_\F \End_B (V_i) \ge 1 & \text{if } i =j.
  \end{cases}
\end{equation}
Therefore, the form $\langle -, - \rangle$ is nondegenerate.  For a commutative ring $R$ and $R$-module $V$, let $V^\vee$ denote the dual space.  Recall that a bilinear form $\langle -, - \rangle \colon V \otimes_R W \to R$ of $R$-modules induces maps
\begin{gather*}
  V \to W^\vee,\ v \mapsto (w \mapsto \langle v,w \rangle), \\
  W \to V^\vee,\ w \mapsto (v \mapsto \langle v,w \rangle).
\end{gather*}
If the form is nondegenerate, then these maps are injective.  If these maps are isomorphisms, then we say that the form is a \define{perfect pairing}.

\begin{eg}
  Let $R=\Z$ and consider the $\Z$-linear form $\Z \otimes_\Z \Z \to \Z$ given by $a \otimes b =2ab$.  This form is nondegenerate but is not a perfect pairing.  However, if $R$ is a field, then any nondegenerate bilinear form on finite-dimensional $R$-modules is a perfect pairing.
\end{eg}

If $\F$ is algebraically closed, then, by Schur's Lemma, we have
\[
  \langle [P_i], [V_j] \rangle = \delta_{i,j},\quad \fa 1 \le i,j \le s,
\]
and the form~\eqref{eq:bilinear-form} is a perfect pairing.

\begin{eg}
  Consider the algebra $D$ of dual numbers.\index{algebra of dual numbers}\index{dual numbers} Let $f \colon D \to D/(x)$ be a homomorphism of $D$-modules.  Since the codomain of $f$ is a simple $D$-module, the kernel of $f$ must be a maximal ideal of $D$.  But the only maximal ideal of $D$ is $(x)$ (see Exercise~\ref{prob:max-ideal-dual-numbers}).  Thus $f$ factors through a map $D/(x) \to D/(x)$.  Since $D/(x)$ is a one-dimensional $\F$-vector space, any such map is simply multiplication by a scalar.  Therefore, $\Hom_D (D, D/(x)) \cong \F$, and so $\langle D, D/(x) \rangle = 1$.
\end{eg}

\Exercises

\medskip

\begin{prob}
  Fill in the details of the claim that the bilinear form~\eqref{eq:bilinear-form} is well-defined.
\end{prob}

\solution{
  Since $F(B\pmd)$ and $F(B\md)$ are both free $\Z$-modules, we can define a bilinear form by specifying its value on pairs of basis elements.  So the form~\eqref{eq:bilinear-form} is well-defined as a bilinear form $\langle -, - \rangle \colon F(B\pmd) \otimes_\Z F(B\md) \to \Z$ (here we use that the dimension of $\Hom_B(P,M)$ depends only on the isomorphism classes of $P$ and $M$).  To show that it descends to a bilinear form $K_0(B) \otimes_\Z G_0(B) \to Z$, we must show that
  \[
    \langle N(B\pmd), F(B\md) \rangle = 0 \quad \text{and} \quad \langle F(B\pmd), N(B\md) \rangle = 0.
  \]
  Now, $N(B\md)$ is generated by elements of the form $[M_1] - [M_2] + [M_3]$ for all short exact sequences $0 \to M_1 \to M_2 \to M_3 \to 0$ in $B\md$.  Consider such a short exact sequence and let $P \in B\pmd$.  Since $P$ is projective, $\Hom_B(P,-)$ is an exact functor.  Thus, we have an exact sequence of $\F$-vector spaces
  \[
    0 \to \Hom_B(P,M_1) \to \Hom_B(P,M_2) \to \Hom_B(P,M_3) \to 0.
  \]
  It follows that $\dim_\F \Hom_B(P,M_2) = \dim_\F \Hom_B(P,M_1) + \dim_\F \Hom_B(P,M_3)$.  Thus
  \begin{align*}
    \langle [P], [M_1] - [M_2] + [M_3] \rangle &= \langle [P],[M_1] \rangle - \langle [P] , [M_2] \rangle + \langle [P], [M_3] \rangle \\
    &= \dim_F \Hom_B(P,M_1) - \dim_F \Hom_B(P,M_2) + \dim_\F \Hom_B(P,M_3) \\
    &=0.
  \end{align*}
  Hence $\langle N(B\pmd), F(B\md) \rangle = 0$.  The proof that $\langle F(B\pmd), N(B\md) \rangle = 0$ is analogous.
}

%
\section{Functors} \label{sec:functors}
%

Suppose $B_1$ and $B_2$ are both unital associative $\F$-algebras and that $\cC_1$ (resp.\ $\cC_2$) is a subcategory of the category of $B_1$-modules (resp.\ $B_2$-modules).  A functor $F \colon \cC_1 \to \cC_2$ is said to be \emph{additive}\index{additive functor}\index{functor!additive} if $F(M \oplus N) \cong F(M) \oplus F(N)$ for all $M,N \in \cC_1$.  If $F$ is additive, then it induces a group homomorphism
\[
  [F] \colon \cK^\spl_0(\cC_1) \to \cK^\spl_0(\cC_2),\ [F]([M]) = [F(M)],\quad \fa M \in \cC_1.
\]
Similarly, if $F$ is exact, then it induces a group homomorphism
\[
  [F] \colon \cK_0(\cC_1) \to \cK_0(\cC_2),\ [F]([M]) = [F(M)],\quad \fa M \in \cC_1.
\]

Fix two finite-dimensional unital associative $\F$-algebras $A$ and $B$.  Recall that an $(A,B)$-bimodule $M$ is, by definition, a left $A$-module and a right $B$-module, where the $A$ and $B$ actions commute.

\begin{eg}
  The algebra $A$ is an $(A,A)$-bimodule via left and right multiplication.  More generally, if $B$ is a subalgebra of $A$, then we can consider $A$ as an $(A,B)$-bimodule, a $(B,A)$-bimodule, or a $(B,B)$-bimodule.
\end{eg}

If $B$ and $C$ are subalgebras of $A$, we let $_BA_C$ denote $A$, considered as a $(B,C)$-bimodule.  If $B$ or $C$ is equal to $A$, we will often omit the corresponding subscript.  So, for example, $A_B$ denotes $A$, considered as an $(A,B)$-bimodule.

Suppose $M$ is an $(A,B)$-bimodule.  Then we have the functor
\[
  M \otimes_B - \colon B\md \to A\md,\quad N \mapsto M \otimes_B N.
\]
Here we consider $M \otimes_B N$ as an $A$-module via the action
\[
  a (m \otimes n) = (am) \otimes n
\]
(and extending by linearity).  The functor $M \otimes_B -$ is always right exact.  It is also left exact (hence exact) precisely when $M$ is projective as a right $B$-module (since projective is the same as flat in our setting -- see Proposition~\ref{prop:fd-alg-modules-props}).  If $M$ is projective as a left $A$-module, then $M \otimes_B -$ maps projective modules to projective modules (see Exercise~\ref{prob:bimodule-functor-left-proj}) and thus restricts to give a functor $B\pmd \to A\pmd$.  Any homomorphism $M \to N$ of $(A,B)$-bimodules gives rise to a natural transformation of functors $M \otimes_B - \to N \otimes_B -$ (see Exercise~\ref{prob:bimodule-hom-natural-trans}).

Now suppose $B$ is a subalgebra of $A$.  Then we have \define{induction} and \define{restriction} functors
\begin{gather*}
  \Ind^A_B \colon B\md \to A\md,\quad \Ind^A_B N = A \otimes_B N,\quad \fa N \in B\md, \\
  \Res^A_B \colon A\md \to B\md,\quad \Res^A_B M = {_BA} \otimes_A M,\quad \fa M \in A\md.
\end{gather*}
If $A$ is projective as a left and right $B$-module, then the above functors are both exact and also induce functors on the corresponding categories of finitely generated projective modules.

\Exercises

\medskip

\begin{prob} \label{prob:bimodule-functor-left-proj}
  Show that if $M$ is an $(A,B)$-bimodule that is projective as a left $A$-module, then the functor $M \otimes_B -$ maps projective $B$-modules to projective $A$-modules.
\end{prob}

\solution{
  Suppose that $M$ is an $(A,B)$-bimodule that is projective as a left $A$-module and that $P$ is a projective $B$-module.  Then there exists a $B$-module $P'$ and nonnegative integer $n$ such that $P \oplus P' \cong B^{\oplus n}$ (as left $B$-modules).  We have isomorphisms of left $A$-modules
  \[
    (M \otimes_B P) \oplus (M \otimes_B P') \cong M \otimes_B (P \oplus P') \cong M \otimes_B B^n \cong M^{\oplus n}.
  \]
  Thus, $M \otimes_B P$ is a summand of the projective left $A$-module $M^{\oplus n}$ and is therefore projective.
}

\begin{prob} \label{prob:bimodule-hom-natural-trans}
  Suppose $f \colon M \to N$ is a homomorphism of $(A,B)$-bimodules.  Use $f$ to define a natural transformation of functors $M \otimes_B - \to N \otimes_B -$.  Prove that what you define is indeed a natural transformation.
\end{prob}

%
\section{Weak categorification}
%

We are now in a position to describe the general idea behind (weak) categorification.  We refer the reader to \cite{KMS09,Maz12} for further discussion and examples.

Suppose $R$ is a commutative ring.  Let $B$ be a unital associative $R$-algebra, and let $\{b_i\}_{i \in I}$ be a fixed generating set for $B$. If $M$ is a $B$-module, then the action of each $b_i$ defines an $R$-linear endomorphism $b_i^M$ of $M$.

\begin{defin}[Naive categorification] \label{def:naive-cat}
  A \define{naive categorification}\index{categorification!naive} of $(B,\{b_i\}_{i \in I}, M)$ is a tuple $(\mathcal{M}, \varphi, \{F_i\}_{i \in I})$, where $\mathcal{M}$ is an abelian category, $\varphi \colon \mathcal{K}_0(\mathcal{M}) \otimes_\Z R \to M$ is an isomorphism, and, for each $i \in I$, $F_i \colon \mathcal{M} \to \mathcal{M}$ is an exact endofunctor of $\mathcal{M}$ such that the following diagram is commutative:
  \[
    \xymatrix{
          \mathcal{K}_0(\mathcal{M}) \otimes_\Z R \ar[r]^{[F_i]} \ar[d]_\varphi & \mathcal{K}_0(\mathcal{M}) \otimes_\Z R \ar[d]_\varphi \\
          M \ar[r]^{b_i^M} & M
        }
  \]
  (We simply write $[F_i]$ above for $[F_i] \otimes \id$.)  In other words, the action of $F_i$ lifts the action of $b_i$.
\end{defin}

The notion of a naive categorification is extremely weak.  We only require that the functors $F_i$ induce the right maps on the level of the Grothendieck group.  A stronger notion would be to categorify the relations amongst the generators $b_i$ (or the induced maps $b_i^M$).  That is, given a set of relations of $B$ (generating all the relations in $B$) we want isomorphisms of functors that descend to these relations in the Grothendieck group.  A naive categorification, together with this extra data, is called a \define{weak categorification}\index{categorification!weak}.

In our examples, the category $\mathcal{M}$ will be a sum of categories of modules.  For example, we will define a collection of finite-dimensional algebras $B_j$, $j \in J$, and set $\mathcal{M}$ equal to $\bigoplus_{j \in J} B_j\md$ (or $\bigoplus_{j \in J} B_j\pmd$).  In this case, the classes of the simple (or projective) $B_j$-modules descend to a distinguished basis of $M$.  For a simple $B_j$-module $V$ and $i \in I$, $[F_i(V)]$ is a sum of classes of simple modules.  Thus, the classes of simple modules give a basis of $M$ for which all the structure coefficients of the action of the basis $\{b_i\}_{i \in I}$ are nonnegative integers.  The existence of a such a distinguished positive integral basis is one of the nicest features of categorification.  Indeed, the knowledge that such a basis of a module $M$ exists is a strong hint that there may be an interesting categorification of $M$.

Note that Definition~\ref{def:naive-cat} (and hence also the notion of a weak categorification) depend on a generating set for the algebra $B$.  In this sense, it is probably more accurate to speak of a categorification of a \emph{presentation} of a module.  However, in some cases it is possible to develop a categorification of a module that is independent of any generating set.  We will see an example of this in Proposition~\ref{prop:fock-space-weak-cat}.

\begin{eg}[{\cite[Ex.~2.8]{Maz12}}] \label{eg:Maz}
  Let $B = \C[b]/(b^2-2b)$ and consider the generating set $\{b\}$.  The relation satisfied by this generator is of course $b^2 = 2b$.  Note that we have rearranged the relation so that all coefficients are positive integers.  Let $M = \C$ be the $B$-module with action given by $b \cdot z = 0$, $z \in M$, and let $N = \C$ be the $B$-module with action given by $b \cdot z = 2z$, $z \in N$.  Let $\cM=\C\md$ be the category of finite-dimensional $\C$-modules and define the functors $F,G \colon \mathcal{M} \to \mathcal{M}$ by
  \begin{gather*}
    F = 0, \quad \text{that is, } F(V) = 0, \quad \fa V \in \cM,\quad \text{and} \\
    G = \id_\cM \oplus \id_\cM, \quad \text{that is, } G(V) = V \oplus V, \quad \fa V \in \cM.
  \end{gather*}
  Define $\varphi \colon \mathcal{K}_0(\cM) \otimes_\Z \C \to M$ and $\psi \colon \mathcal{K}_0(\cM) \otimes_\Z \C \to N$ both by $z[\C] \mapsto z$ (where $[\C]$ here denotes the class of the simple one-dimensional $\C$-module).

  Since
  \begin{gather*}
    \varphi \circ [F] (z[\C]) = 0 = b \cdot \varphi(z [\C]),\quad \fa z \in \C,\quad \text{and} \\
    \psi \circ [G] (z[\C]) = \psi (z [G(\C)]) = \psi (z[\C \oplus \C]) = \psi(2z[\C]) = 2z = b \cdot z = b \cdot \psi (z[\C]),
  \end{gather*}
  we see that $(\cM,\varphi,F)$ and $(\cM,\psi,G)$ are naive categorifications of $(B,\{b\},M)$ and $(B,\{b\},N)$, respectively.  It is easy to verify (see Exercise~\ref{prob:Maz}) that we have isomorphisms of functors
  \begin{equation} \label{eq:Maz-isoms}
    F \circ F \cong F \oplus F,\quad G \circ G \cong G \oplus G.
  \end{equation}
  Thus, as operators on $\mathcal{K}_0(\cM)$, we have $[F]^2 = 2[F]$ and $[G]^2=2[G]$.  So the isomorphisms~\eqref{eq:Maz-isoms} lift the relation $b^2=2b$.  Therefore, $(\cM, \varphi, F)$ and $(\cM, \psi, G)$ are weak categorifications of $(B,\{b\},M)$ and $(B,\{b\},N)$, respectively.  Note that $B \cong \C[S_2]$ via the map $b \mapsto s_1+1$ and, under this isomorphism, the modules $M$ and $N$ become the sign and trivial $\C[S_2]$-modules, respectively.
\end{eg}

\iftoggle{solutions}{}{\newpage}
\Exercises

\medskip

\begin{prob} \label{prob:Maz}
  Verify that the functors $F$ and $G$ of Example~\ref{eg:Maz} satisfy~\eqref{eq:Maz-isoms}.
\end{prob}

\begin{prob}
  Suppose $R$ is a ring that is free as an abelian group, and let $\mathbf{r} = \{r_i\}_{i \in I}$ be a basis of $R$ such that the multiplication in this basis has nonnegative integer coefficients: $r_i r_j = \sum_k c_{ij}^k r_k$, $c_{ij}^k \in \N$.   Let $M$ be a left $R$-module with a basis $\mathbf{b} = \{b_j\}_{j \in J}$ such that $r_i b_j = \sum_k d_{ij}^k b_k$, $d_{ij}^k \in \N$.  By defining $\mathcal{M}$ to be an appropriate sum of copies of the category of finite-dimensional $\F$-vector spaces, show that one can always define a (rather trivial) weak categorification of $(R,\mathbf{r},M)$.
\end{prob}

\solution{
  See \cite[\S2.1, p.~482]{KMS09}.
}

%
\chapter{Categorification of the polynomial representation of the Weyl algebra}
\chaptermark{The Weyl algebra}
\pagestyle{headings}
%

In this chapter will discuss a categorification that is more sophisticated than the simple examples we have seen so far.  In particular, we will use the categories of modules over nilcoxeter algebras to categorify the polynomial representation of the Weyl algebra.  The action of the Weyl algebra will be categorified by induction and restriction functors.  This categorification was first carried out by Khovanov in~\cite{Kho01}.  We refer the reader to that reference for further details.  As before, we fix an arbitrary field $\F$ and assume that all modules are finitely generated.

%
\section{The Weyl algebra} \label{sec:Weyl-algebra}
%

The Weyl algebra is the algebra of differential operators with polynomial coefficients in one variable.  For our purposes, we define it as follows.

\begin{defin}[Weyl algebra]
  The \define{Weyl algebra} $W$ is the unital associative algebra over $\Z$ with generators $x, \partial$ and defining relation $\partial x = x \partial + 1$.
\end{defin}

Let $R_\Q = \Q[x]$ be the $\Q$-vector space spanned by $x^0, x^1, x^2,\dotsc$.  There is a natural action of $W$ on $\R_\Q$ given by
\[
  x \cdot x^n = x^{n+1},\quad \partial \cdot x^n = n x^{n-1},\quad \fa n=0,1,2,\dotsc.
\]
The abelian subgroups
\[
  R = \Span_\Z \{x^n/n!\}_{n=0}^\infty \quad \text{and} \quad R' = \Span_\Z \{x^n\}_{n=0}^\infty
\]
of $R_\Q$ are easily seen to be $W$-submodules of $R_\Q$.  The module $R_\Q$ is important for several reasons.  In particular, it is faithful and, as a $(W \otimes_\Z \Q)$-module, it is irreducible.  It is also the unique $(W \otimes_\Z \Q)$-module generated by an element (the element $x^0 \in R_\Q$) annihilated by $\partial$.  This can be proven directly or it can be seen to follow from a generalized Stone--von Neumann Theorem (see~\cite[Th.~2.11]{SY13}).

We define a $\Q$-valued symmetric bilinear form on $R_\Q$ by
\[
  \langle x^n, x^m \rangle = \delta_{n,m}\, n!.
\]
This form restricts to a $\Z$-valued perfect pairing $\langle -, - \rangle \colon R' \times R \to \Z$.

%
\section{The nilcoxeter algebra and its modules} \label{sec:nilcoxeter}
%

We first recall some basic facts about the symmetric group.  The group algebra $\F[S_n]$ of the symmetric group $S_n$ on $n$ letters is generated by the simple transpositions $s_i = (i,i+1)$.  A complete set of relations for these generators is
\begin{gather*}
  s_i^2 = 1 \text{ for } i=1,2,\dotsc,n-1, \\
  s_is_j = s_js_i \text{ for } i,j = 1,\dotsc,n-1 \text{ such that } |i-j|>1, \\
  s_is_{i+1}s_i = s_{i+1}s_is_{i+1} \text{ for } i=1,2,\dotsc,n-2.
\end{gather*}
(These are simply the relations for the transpositions in the symmetric group.)  The last two sets of relations are known as the \emph{braid relations}\index{braid relation}.  Any elements of $S_n$ can be written as a product
\begin{equation} \label{eq:Sn-element-product-transpositions}
  \sigma = s_{i_1} s_{i_2} \dotsm s_{i_k}.
\end{equation}
If $k$ is minimal among such expressions, then it is called the \emph{length}\index{length of a permutation} of $\sigma$ and is denoted $\ell(\sigma)$.  (The length of a permutation is also equal to the number of inversions it creates.)  Any expression~\eqref{eq:Sn-element-product-transpositions} of minimal length (i.e.\ where $k = \ell(\sigma)$) is called a \define{reduced expression} for $\sigma$.  Any reduced expression for $\sigma \in S_n$ can be obtained from any other reduced expression for $\sigma$ by a sequence of braid relations.  If~\eqref{eq:Sn-element-product-transpositions} is not a reduced expression, then one may use the braid relations to replace it by an expression in which two $s_j$ (for some $j=1,\dotsc,n-1$) appear immediately next to one another.  Using the relation $s_j^2=1$, the length of the expression can be reduced by two.  Continuing in this manner, one may obtain a reduced expression from any (potentially non-reduced) expression.  The element of $S_n$ mapping $i$ to $n-i+1$ (for $i=1,\dotsc,n$) is the unique maximal length element of $S_n$.  Its length is $n(n-1)/2$.

\begin{defin}[Nilcoxeter algebra]
  Fix a nonnegative integer $n$.  The \define{nilcoxeter algebra} $N_n$ is the unital $\F$-algebra generated by $u_1,\dotsc,u_{n-1}$ subject to the relations
  \begin{gather*}
    u_i^2 = 0 \text{ for } i=1,2,\dotsc,n-1, \\
    u_iu_j = u_ju_i \text{ for } i,j = 1,\dotsc,n-1 \text{ such that } |i-j|>1, \\
    u_iu_{i+1}u_i = u_{i+1}u_iu_{i+1} \text{ for } i=1,2,\dotsc,n-2.
  \end{gather*}
  By convention, we set $N_0=N_1=\F$.
\end{defin}

Note that the nilcoxeter algebra is quite similar to $\F[S_n]$.  The only difference is that the generators $u_i$ square to zero instead of one.  The proof of the following lemma is left as an exercise (Exercise~\ref{prob:nilcoxeter-nilpotence}).

\begin{lem} \label{lem:nilcoxeter-nilpotence}
  If $k > n(n-1)/2$, then $u_{i_1} u_{i_2} \dotsm u_{i_k} = 0$ for all $i_1,i_2,\dotsc,i_k \in \{1,\dotsc,n-1\}$.
\end{lem}

Since the relations defining $N_n$ are homogenous in the $u_i$, we can define an $\N$-grading $N_n = \bigoplus_{m \in \N} N_n^{(m)}$ on $N_n$ (as an algebra, i.e., $N_n^{(m_1)} N_n^{(m_2)} \subseteq N_n^{(m_1+m_2)}$ for all $m_1,m_2 \in \N$) by setting the degree of $u_i$ to be one for $i=1,\dotsc,n-1$.  Let $I = \bigoplus_{m \ge 1} N_n^{(m)}$ be the sum of the (strictly) positively graded pieces of $N_n$.  In other words, $I$ is the ideal of $N_n$ generated by the $u_i$, $i=1,\dotsc,n-1$.  This is a maximal ideal of $N_n$ since $N_n/I$ is a one-dimensional $N_n$-module (spanned by the image of the unit of $N_n$) and hence simple.  It follows from Lemma~\ref{lem:nilcoxeter-nilpotence} that $I^k=0$ for $k > n(n-1)/2$.

\begin{prop} \label{prop:nilcoxeter-modules}
  The nilcoxeter algebra $N_n$ has a unique simple module, denoted $L_n$.  This is the one-dimensional module on which all $u_i$, $i=1,\dotsc,n-1$, act by zero.  The projective cover of $L_n$ is $N_n$.
\end{prop}

\begin{proof}
  Let $V$ be a simple $N_n$-module.  Then $IV$ is a submodule of $V$.  Thus, we must have $IV = V$ or $IV=0$.  If $IV = V$, then, for $k > n(n-1)/2$, we would have
  \[
    V = IV = I^2V = \dotsb = I^kV = 0,
  \]
  contradicting the fact that $V \ne 0$.  Thus $IV=0$ and so $V=L_n$.

  Since $N_n$ is a free (hence projective) $N_n$-module, to show that it is the projective cover of $L_n$, it suffices to show that the kernel $I$ of the map $N_n \to N_n/I$ is a superfluous module of $N_n$.  If $I + H = N_n$ for some submodule (i.e.\ ideal) $H$ of $N_n$, then $H$ must contain an element of the form $1 - a$, with $a \in I$.  This element is invertible, with inverse $1+a+a^2 + \dotsb + a^k$, where $k$ is any integer greater than $n(n-1)/2$.  Hence, $H=N_n$.  So $I$ is superfluous as desired.
\end{proof}

By Proposition~\ref{prop:nilcoxeter-modules}, we have
\[
  G_0(N_n) = \Z[L_n],\quad K_0(B) = \Z[N_n].
\]
Let\index{GN@$\cG_N$}\index{KN@$\cK_N$}
\begin{gather*}
  \cG_N = \cK_0\left( \bigoplus_{n=0}^\infty N_n\md \right) = \bigoplus_{n=0}^\infty G_0(N_n) = \bigoplus_{n=0}^\infty \Z[L_n], \quad \text{and} \\
  {\cK_N} = \cK_0\left( \bigoplus_{n=0}^\infty N_n\pmd \right) = \bigoplus_{n=0}^\infty K_0(N_n) = \bigoplus_{n=0}^\infty \Z[N_n].
\end{gather*}
We define a bilinear form $\langle -, - \rangle \colon \cG_N \otimes_\Z {\cK_N} \to \Z$ by declaring $G_0(N_n)$ to be orthogonal to $K_0(N_m)$ for $n \ne m$, and using the form~\eqref{eq:bilinear-form} when $n = m$.

Define isomorphisms of $\Z$-modules
\begin{gather*}
  \varphi_{\cG_N} \colon {\cG_N} \to R = \Span_\Z \{x^n/n!\}_{n=0}^\infty,\quad [L_n] \mapsto x^n/n!, \\
  \varphi_{\cK_N} \colon {\cK_N} \to R' = \Span_\Z \{x^n\}_{n=0}^\infty,\quad [N_n] \mapsto x^n.
\end{gather*}
Since
\[
  \langle x^m, x^n/n! \rangle = \delta_{m,n} = \langle [N_m], [L_n] \rangle,
\]
we see that the above maps respect the bilinear forms we have defined on the spaces involved.  In other words, we have
\[
  \langle a, b \rangle = \langle \varphi_{\cK_N}(a), \varphi_{\cG_N}(b) \rangle, \quad \fa a \in {\cK_N},\ b \in {\cG_N}.
\]
Our next goal is to categorify the action of the Weyl algebra $W$ on the modules $R$ and $R'$.

\Exercises

\medskip

\begin{prob} \label{prob:nilcoxeter-nilpotence}
  Prove Lemma~\ref{lem:nilcoxeter-nilpotence}.
\end{prob}

\solution{
  Since $k$ is greater than the length of the longest element of $S_n$, the expression $s_{i_1} s_{i_2} \dotsm s_{i_k}$ cannot be reduced.  Because the nilcoxeter generators satisfy the braid relations, we have, as in the symmetric group, $u_{i_1} \dotsm u_{i_k} = u_{j_1} \dotsm u_{j_k}$ for some $j_1,\dotsc,j_k \in \{1,\dotsc,n-1\}$ with $j_\ell = j_{\ell+1}$ for some $\ell \in \{1,\dotsc,n-2\}$.  Then $u_{j_\ell} u_{j_\ell+1}=0$ and the result follows.
}

%
\section{Weak categorification of the polynomial representation} \label{sec:Weyl-weak-cat}
%

We can view the nilcoxeter algebra $N_n$ naturally as the subalgebra of $N_{n+1}$ generated by $u_1,\dotsc,u_{n-1}$.  For each $n \in \N$, define
\[
  X_n = (N_{n+1})_{N_n},\quad D_n = {_{N_n}}N_{n+1}.
\]
In other words, $X_n$ is $N_{n+1}$, viewed as an $(N_{n+1},N_n)$-bimodule, and $D_n$ is $N_{n+1}$, viewed as an $(N_n,N_{n+1})$-bimodule.  Thus
\begin{gather*}
  \left( X_n \otimes_{N_n} - \right) = \Ind_{N_n}^{N_{n+1}} \colon N_n\md \to N_{n+1}\md, \\
  \left( D_n \otimes_{N_{n+1}} - \right) = \Res^{N_{n+1}}_{N_n} \colon N_{n+1}\md \to N_n\md.
\end{gather*}

For $\sigma \in S_n$, let $\sigma = s_{i_1} \dotsm s_{i_k}$ be a reduced expression and define
\[
  u_\sigma = u_{i_1} \dotsm u_{i_k}.
\]
It follows from the discussion in Section~\ref{sec:nilcoxeter} that $u_\sigma$ is independent of the reduced expression for $\sigma$.  It also follows that $\{u_\sigma\}_{\sigma \in S_n}$ is a basis for $N_n$ and the multiplication in this basis is given by
\[
  u_\sigma u_\tau =
  \begin{cases}
    u_{\sigma \tau} & \text{if } \ell(\sigma\tau) = \ell(\sigma) + \ell(\tau), \\
    0 & \text{otherwise}.
  \end{cases}
\]
\comments{Maybe I should have included this in the section that defined the nilcoxeter algebras.}

We would like the functors of tensoring with $D_n$ and $X_n$ to be exact and map projectives to projectives.  As we saw in Section~\ref{sec:functors}, we need the following lemma.

\begin{lem}[{\cite[Prop.~4]{Kho01}}] \label{lem:XD-projective}
  The bimodules $X_n$ and $D_n$ are both left and right projective for all $n \in \N$.
\end{lem}

\begin{proof}
  As a left $N_{n+1}$-module, $X_n$ is free of rank one, hence projective.  We next show that $X_n$ is right projective.   Let $\sigma \in S_{n+1}$ and set $i = \sigma(n+1)$.  Then $\sigma$ can be written uniquely in the form
  \[
    \sigma = s_i s_{i+1} \dotsm s_n \sigma',\quad \sigma' \in S_n
  \]
  (in particular, $\sigma' = s_n \dotsm s_{i+1} s_i \sigma$).  Thus, $X_n$ is a free right $N_n$-module with basis
  \[
    1,\ u_n,\ u_{n-1} u_n,\ \dotsc,\ u_1 u_2 \dotsm u_n.
  \]
  Hence $X_n$ is projective as a right $N_n$-module.  The argument that $D_n$ is left and right projective is analogous.
\end{proof}

\begin{cor}
  The functors $X_n \otimes_{N_n} -$ and $D_n \otimes_{N_{n+1}} -$ are exact and induce functors
  \begin{gather*}
    \left( X_n \otimes_{N_n} - \right) = \Ind_{N_n}^{N_{n+1}} \colon N_n\pmd \to N_{n+1}\pmd, \\
    \left( D_n \otimes_{N_{n+1}} - \right) = \Res^{N_{n+1}}_{N_n} \colon N_{n+1}\pmd \to N_n\pmd.
  \end{gather*}
\end{cor}

Let $N = \bigoplus_{n=0}^\infty  N_n$ (sum of algebras).  Then $N$ is an associative algebra.  However, it is no longer unital.  Instead, it has an infinite family $1_{N_n}$, $n \in \N$, of pairwise orthogonal idempotents.  Any $N_n$-module $M$, for $n \in \N$, is naturally an $N$-module.  Namely, we set $aM=0$ for all $a \in N_m$ with $m \ne n$.  In the same way, any $(N_n,N_m)$-bimodule, $m,n \in \N$, can be viewed as an $(N,N)$-bimodule.  Define the $(N,N)$-bimodules
\[
  X = \bigoplus_{n=0}^\infty X_n \quad \text{and} \quad D = \bigoplus_{n=0}^\infty D_n.
\]
(Since we do not refer to the algebra of dual numbers in this chapter, we hope there will be no confusion in our use of the same notation $D$.)

Let
\[
  \cN=\bigoplus_{n=0}^\infty N_n\md \quad \text{and} \quad \cN_\pj = \bigoplus_{n=0}^\infty N_n\pmd.
\]
Then $\cN$ and $\cN_\pj$ can be naturally viewed as full subcategories of the category of finite-dimensional $N$-modules.  If we define
\[
  \Ind = \bigoplus_{n=0}^\infty \Ind_{N_n}^{N_{n+1}} \quad \text{and} \quad
  \Res = \bigoplus_{n=0}^\infty \Res^{N_{n+1}}_{N_n},
\]
(as endofunctors of $\cN$ or $\cN_\pj$) then we have isomorphisms of functors
\[
  X \otimes_N - \cong \Ind \quad \text{and} \quad D \otimes_N - \cong \Res.
\]
Recall that ${\cG_N} = \cK_0(\cN)$ and ${\cK_N} = \cK_0(\cN_\pj)$.

\begin{prop} \label{prop:Weyl-naive-cat}
  The tuples $(\cN,\varphi_{\cG_N},\{\Ind,\Res\})$ and $(\cN_\pj,\varphi_{\cK_N},\{\Ind,\Res\})$ are naive categorifications of $(W,\{x,\partial\},R)$ and $(W,\{x,\partial\},R')$, respectively.
\end{prop}

\begin{proof}
  We have already seen in Section~\ref{sec:nilcoxeter} that $\varphi_{\cG_N} \colon {\cG_N} \to R$ and $\varphi_{\cK_N} \colon {\cK_N} \to R'$ are $\Z$-module isomorphisms.  We saw in the proof of Lemma~\ref{lem:XD-projective} that $X_n$ is a free right $N_n$-module of rank $n+1$.  Thus, since $\dim L_n=1$, we have
  \[
    \dim_\F \Ind(L_n) = \dim_\F (X_n \otimes_{N_n} L_n) = n+1.
  \]
  Therefore, the composition series of $X_n \otimes_{N_n} L_n$ must have the unique simple (one-dimensional) $N_{n+1}$-module $L_{n+1}$ occur with multiplicity $n+1$.  So we have
  \[
    [\Ind(L_n)] = (n+1)[L_{n+1}] \in {\cG_N}.
  \]
  We also have (as left $N_{n+1}$-modules)
  \[
    \Ind(N_n) = X_n \otimes_{N_n} N_n = N_{n+1} \otimes_{N_n} N_n \cong N_{n+1},
  \]
  and so
  \[
    [\Ind(N_n)] = [N_{n+1}] \in {\cK_N}.
  \]
  Furthermore,
  \[
    \dim_\F \Res(L_{n+1}) = \dim_\F (D_n \otimes_{N_{n+1}} L_{n+1}) = \dim_\F (N_{n+1} \otimes_{N_{n+1}} L_{n+1}) = 1.
  \]
  Thus $\Res(L_{n+1}) = L_n$ and so
  \[
    [\Res(L_{n+1})] = [L_n] \in {\cG_N}.
  \]
  Finally, we have (as left $N_n$-modules)
  \[
    \Res(N_{n+1}) = D_n \otimes_{N_{n+1}} N_{n+1} = {_{N_n} N_{n+1}} \otimes_{N_{n+1}} N_{n+1} = {_{N_n} N_{n+1}} \cong N_n^{\oplus (n+1)},
  \]
  where the last isomorphism follows from the fact that $N_{n+1}$ is free of rank $n+1$ as a left $N_n$-module, as in the proof of Lemma~\ref{lem:XD-projective}.  Thus we have
  \[
    [\Res(N_{n+1})] = (n+1)[N_n] \in {\cK_N}.
  \]

  It follows from the above computations that we have, for all $n \in \N$,
  \begin{gather*}
    \varphi_{\cG_N} \circ [\Res] ([L_{n+1}]) = \varphi_{\cG_N}([L_n]) = x^n/n! = \partial \cdot x^{n+1}/(n+1)! = \partial \circ \varphi_{\cG_N} ([L_{n+1}]), \\
    \varphi_{\cG_N} \circ [\Ind] ([L_n]) = \varphi_{\cG_N}((n+1)[L_{n+1}]) = (n+1) x^{n+1}/(n+1)! = x \cdot x^n/n! = x \circ \varphi_{\cG_N}([L_n]), \\
    \varphi_{\cK_N} \circ [\Res] ([N_{n+1}]) = \varphi_{\cK_N}((n+1)[N_n]) = (n+1) x^n = \partial \cdot x^{n+1} = \partial \circ \varphi_{\cK_N} ([N_{n+1}]), \\
    \varphi_{\cK_N} \circ [\Ind] ([N_n]) = \varphi_{\cK_N} ([N_{n+1}]) = x^{n+1} = x \cdot x^n = x \circ \varphi_{\cK_N} ([N_n]).
  \end{gather*}
  In other words, we have the following commutative diagrams:
  \[
    \xymatrix{
      {\cG_N} \ar[r]^{[\Ind]} \ar[d]_{\varphi_{\cG_N}} & {\cG_N} \ar[d]^{\varphi_{\cG_N}} \\
      R \ar[r]^x & R
    }
    \qquad
    \xymatrix{
      {\cG_N} \ar[r]^{[\Res]} \ar[d]_{\varphi_{\cG_N}} & {\cG_N} \ar[d]^{\varphi_{\cG_N}} \\
      R \ar[r]^\partial & R
    }
    \qquad
    \xymatrix{
      {\cK_N} \ar[r]^{[\Ind]} \ar[d]_{\varphi_{\cK_N}} & {\cK_N} \ar[d]^{\varphi_{\cK_N}} \\
      R \ar[r]^x & R
    }
    \qquad
    \xymatrix{
      {\cK_N} \ar[r]^{[\Res]} \ar[d]_{\varphi_{\cK_N}} & {\cK_N} \ar[d]^{\varphi_{\cK_N}} \\
      R \ar[r]^\partial & R
    }
  \]
  The result follows.
\end{proof}

We would like to strengthen this naive categorification to a weak categorification.  For this, we need an isomorphism of functors lifting the defining relation of the Weyl algebra.

\begin{prop} \label{prop:nilcoxter-bimodule-isom}
  For each $n \in \N$, we have an isomorphism of $(N_n,N_n)$-bimodules
  \[
    D_{n+1} \otimes_{N_{n+1}} X_n \cong (X_{n-1} \otimes_{N_{n-1}} D_n) \oplus N_n,
  \]
  where $N_n$ is considered as an $(N_n,N_n)$-bimodule in the usual way (via left and right multiplication).  We thus have an isomorphism of $(N,N)$-bimodules
  \[
    D \otimes_N X \cong (X \otimes_N D) \oplus N.
  \]
\end{prop}

\begin{proof}
  We have isomorphisms of $(N_n,N_n)$-bimodules
  \[
    D_{n+1} \otimes_{N_{n+1}} X_n \cong {_{N_n} (N_{n+1})_{N_n}},\quad X_{n-1} \otimes_{N_{n-1}} D_n \cong N_n \otimes_{N_{n-1}} N_n.
  \]
  Let
  \[
    m_1 \colon N_n \hookrightarrow N_{n+1}
  \]
  be the natural inclusion of $(N_n,N_n)$-bimodules (i.e.\ uniquely determined by $1 \mapsto 1$).  We also have an injective homomorphism of $(N_n,N_n)$-bimodules
  \[
    m_2 \colon N_n \otimes_{N_{n-1}} N_n \hookrightarrow N_{n+1},\quad m_2(a \otimes b) = a u_n b,\ a,b \in N_n
  \]
  (see Exercise~\ref{prob:m2-map}).  For $\sigma \in S_{n+1}$, we have $u_\sigma \in m_1(N_n)$ if and only if $\sigma(n+1)=n+1$.  If $\sigma(n+1) \ne n+1$, then we can write $\sigma = \tau_1 s_n \tau_2$ for $\tau_1,\tau_2 \in S_n$.  Hence $u_\sigma \in m_2(N_n \otimes_{N_{n-1}} N_n)$.  Therefore, $m_1$ and $m_2$ define an $(N_n,N_n)$-bimodule homomorphism
  \[
    (N_n \otimes_{N_{n-1}} N_n) \oplus N_n \cong {_{N_n}(N_{n+1})}_{N_n}
  \]
  as desired.
\end{proof}

\begin{cor}
  We have isomorphisms of endofunctors of $N_n\md$ (hence also of $N_n\pmd$)
  \[
    \Res^{N_{n+1}}_{N_n} \circ \Ind_{N_n}^{N_{n+1}} \cong \left(\Ind_{N_{n-1}}^{N_n} \circ \Res^{N_n}_{N_{n-1}}\right) \oplus \id,
  \]
  and hence isomorphisms of endofunctors of ${\cG_N}$ (thus also of ${\cK_N}$)
  \begin{equation} \label{eq:Weyl-functor-isom}
    \Res \circ \Ind \cong (\Ind \circ \Res) \oplus \id.
  \end{equation}
\end{cor}

\begin{proof}
  This follows from the fact that
  \begin{gather*}
    \left( D_{n+1} \otimes_{N_{n+1}} X_n \right) \otimes_{N_n} - \cong \Res^{N_{n+1}}_{N_n} \circ \Ind_{N_n}^{N_{n+1}}, \\
    (X_{n-1} \otimes_{N_{n-1}} D_n) \otimes_{N_n} - \cong \Ind_{N_{n-1}}^{N_n} \circ \Res^{N_n}_{N_{n-1}}, \\
    N_n \otimes_{N_n} - \cong \id. \qedhere
  \end{gather*}
\end{proof}

The isomorphism~\eqref{eq:Weyl-functor-isom} categorifies the defining relation $\partial x = x \partial + 1$.  Together with Proposition~\ref{prop:Weyl-naive-cat}, this shows that we have a weak categorification of the modules $R$ and $R'$ of the Weyl algebra $W$.

Since induction is left adjoint to restriction (see, for example, \cite[(2.19)]{CR81}), we have
\[
  \Hom_\cN(\Ind(P),M) \cong \Hom_\cN(P,\Res(M)), \quad \fa P \in \cN_\pj,\ M \in \cN.
\]
Thus
\[
  \langle [\Ind] (a), b \rangle = \langle a, [\Res] (b) \rangle, \quad \fa a \in {\cK_N},\ b \in {\cG_N}.
\]
So we have a categorification of the fact that $x$ is adjoint to $\partial$, that is,
\[
  \langle x \cdot f, g \rangle = \langle f, \partial \cdot g \rangle, \quad \fa f,g \in R_\Q.
\]
Note that one also has $\langle \partial \cdot f, g \rangle = \langle f, x \cdot g \rangle$ for all $f,g \in R_\Q$, and so one might guess that $\Ind$ is also right adjoint (hence biadjoint) to $\Res$.  In fact, this is not quite true.  However, $\Ind$ is \emph{twisted} right adjoint to $\Res$, and, in this case, this is enough to yield the right relationship on the level of Grothendieck groups.  We refer the reader to~\cite[\S2.4]{Kho01} for more details.  Biadjointness of functors is often a requirement for \emph{strong} categorification.  This is one reason why the weak categorification we have described in this chapter has not yet been lifted to a strong categorification.

\Exercises

\medskip

\begin{prob} \label{prob:m2-map}
  Show that the map $m_2$ in the proof of Proposition~\ref{prop:nilcoxter-bimodule-isom} is well defined.
\end{prob}

\begin{prob}
  Show that $\Ind$ is not right adjoint to $\Res$.
\end{prob}

\comments{Could add another section here on additional structure:
\begin{itemize}
  \item grading on $\Q[x]$ by degree corresponds to decomposition of $\cN$ into $N_n\md$'s.
  \item considering graded modules gives quantum Weyl group
\end{itemize}
}

%
\chapter{Categorification of the Fock space representation of the Heisenberg algebra} \label{ch:heisenberg}
\chaptermark{The Heisenberg algebra}
\pagestyle{headings}
%

We now turn our attention to the categories of modules over the group algebras of the symmetric groups.  (In fact, the results of this section go through for Hecke algebras at generic parameters, but we will stick to the more familiar setting of symmetric groups.)  We will see that these categories provide a categorification of the Heisenberg algebra and its Fock space representation.  We will see later that this can be lifted to a \emph{strong} categorification.  Due to time constraints, we will not work out all the explicit calculations for modules for the symmetric groups as we did for the nilcoxeter algebras.  Instead, we will state the results of these calculations and give references to the literature where the relevant details can be found.

%
\section{Symmetric functions} \label{sec:sym}
%

We begin by recalling some basic facts about symmetric functions.  We refer the reader to \cite[Ch.~I]{Mac95} for further details.  For $n \in \N$, let $\cP(n)$\index{Pn@$\cP(n)$} denote the set of partitions of $n$ and define $\cP = \bigcup_{n \in \N} \cP(n)$\index{P@$\cP$}.  Let $\Sy$\index{Sym@$\Sy$} denote the algebra of symmetric functions in countably many variables $x_1,x_2,\dotsc$ over $\Z$.  This is a graded algebra:
\[
  \Sy = \bigoplus_{n \in \N} \Sy_n,
\]
where $\Sy_n$ is the $\Z$-submodule of $\Sy$ consisting of homogeneous polynomials of degree $n$.  By convention, we set $\Sy_n = 0$ for $n < 0$.

For $\lambda \in \cP$, we define the \define{monomial symmetric function}
\[
  m_\lambda = \sum_{\alpha \text{ a rearrangement of } \lambda} x^\alpha,
\]
where $x^\alpha = x_1^{\alpha_1} x_2^{\alpha_2} \dotsm$ and we view $\lambda$ as a partition with infinitely many parts, all but finitely many of which are zero.  Then $\{m_\lambda\}_{\lambda \in \cP}$ is a $\Z$-basis for $\Sy$.  For $n \in \N$, we define the \define{complete symmetric function} $h_n$, the \define{elementary symmetric function} $e_n$, and the \define{power sum symmetric function} $p_n$ by
\[
  h_n = \sum_{\lambda \in \cP(n)} m_\lambda,\quad e_n = m_{(1^n)},\quad p_n = m_{(n)}.
\]
Then we define
\[
  h_\lambda = h_{\lambda_1} \dotsm h_{\lambda_\ell},\quad e_\lambda = e_{\lambda_1} \dotsm e_{\lambda_\ell},\quad p_\lambda = p_{\lambda_1} \dotsm p_{\lambda_\ell}, \quad \fa \lambda = (\lambda_1,\lambda_2,\dotsc,\lambda_\ell) \in \cP.
\]
Then $\{h_\lambda\}_{\lambda \in \cP}$ and $\{e_\lambda\}_{\lambda \in \cP}$ are $\Z$-bases for $\Sy$.  In other words, $\{h_n\}_{n \in \N_+}$ and $\{e_n\}_{n \in \N_+}$ are sets of polynomial generators for $\Sy$:
\[
  \Sy = \Z[h_1,h_2,\dotsc],\quad \Sy = \Z[e_1,e_2,\dotsc].
\]
On the other hand, $\{p_\lambda\}_{\lambda \in \cP}$ is only a $\Q$-basis for $\Sy \otimes_\Z \Q$.  That is,
\[
  \Sy \otimes_\Z \Q = \Q[p_1,p_2,\dotsc].
\]

We define an inner product $\langle -, - \rangle$ on $\Sy$ by declaring the monomial and complete symmetric functions\index{monomial symmetric function}\index{complete symmetric function} to be dual to each other:
\[
  \langle m_\lambda, h_\mu \rangle = \delta_{\lambda,\mu},\quad \fa \lambda, \mu \in \cP.
\]

The \define{Schur functions} are given by
\[
  s_\lambda = \det (h_{\lambda_i-i+j})_{1 \le i, j \le n},\quad \lambda \in \cP,
\]
where $n$ is greater than the length of the partition $\lambda$.  In particular, we have
\[
  s_{(n)} = h_n,\quad s_{(1^n)} = e_n,\quad \fa n \in \N_+.
\]
The Schur functions are self-dual:
\[
  \langle s_\lambda, s_\mu \rangle = \delta_{\lambda,\mu},\quad \fa \lambda, \mu \in \cP.
\]

\begin{defin}[Hopf algebra]
  A \define{Hopf algebra} over $\Z$ is a tuple $(B,\nabla,\eta,\Delta,\varepsilon,S)$ such that
  \begin{enumerate}[(a)]
    \item $B$ is a $\Z$-module;
    \item $\nabla \colon B \otimes B \to B$ (the \define{multiplication map}) and $\eta \colon \Z \to B$ (the \define{unit}) are $\Z$-linear maps making $(B,\nabla,\eta)$ a unital associative algebra;
    \item $\Delta \colon B \to B \otimes B$ (the \define{comultiplication map}) and $\varepsilon \colon B \to \Z$ (the \define{counit}) are $\Z$-linear maps making $(B,\Delta,\varepsilon)$ a counital coassociative coalgebra;
    \item $S \colon B \to B$ is a $\Z$-linear map (the \define{antipode});
    \item the maps $\nabla$, $\eta$, $\Delta$, $\varepsilon$, and $S$ satisfy certain compatibility conditions.  In particular, the coproduct $\Delta$ is an algebra homomorphism.
  \end{enumerate}
  We refer the reader to the many books on Hopf algebras for a more precise definition and further discussion.  If we omit the antipode from the definition, then we are left with the definition of a \define{bialgebra}.
\end{defin}

\begin{rem}
  The definition of a coalgebra is the dual of the definition of an algebra -- one simply reverses all the arrows in the definition (after stating the definition completely in terms of morphisms).  The structure of a Hopf algebra is precisely the structure one needs in order to define tensor products of modules and duals of modules.  In particular, if $M$ and $N$ are two modules for a Hopf algebra $B$, then the tensor product is a $B$-module via the composition
  \[
    B \xrightarrow{\Delta} B \otimes B \to (\End_\Z M) \otimes (\End_\Z N) \to \End_\Z (M \otimes N).
  \]
  Similarly, the antipode is used to define a $B$-module structure on the dual of a $B$-module.
\end{rem}

\begin{defin}[Graded connected Hopf algebra]
  We say that a bialgebra $H$ over $\Z$ is \emph{graded}\index{graded bialgebra}\index{bialgebra!graded} if $H = \bigoplus_{n \in \N} H_n$, where each $H_n$, $n \in \N$, is finitely generated and free as a $\Z$-module, and the following conditions are satisfied:
  \begin{gather*}
    \nabla(H_k \otimes H_\ell) \subseteq H_{k + \ell},\quad \Delta(H_k) \subseteq \bigoplus_{j=0}^k H_j \otimes H_{k-j},\quad k,\ell \in \N, \\
    \eta(\Z) \subseteq H_0,\quad \varepsilon(H_k) = 0 \text{ for } k \in \N_+.
  \end{gather*}
  We say that $H$ is \emph{graded connected}\index{graded connected bialgebra}\index{bialgebra!graded connected} if it is graded and $H_0 = \Z 1_H$.  A graded connected bialgebra is a Hopf algebra with invertible antipode (see, for example, \cite[p.~389, Cor.~5]{Hand08}) and thus we will also call such an object a \define{graded connected Hopf algebra}\index{Hopf algebra!graded connected}.
\end{defin}

The algebra $\Sy$ is in fact a graded connected Hopf algebra.  The coproduct $\Delta \colon \Sy \to \Sy \otimes_\Z \Sy$ is given on the elementary monomial and homogeneous symmetric functions by
\begin{equation} \label{eq:sym-coproduct}
  \Delta(e_n) = \sum_{i=0}^n e_i \otimes e_{n-i},\quad \Delta(h_n) = \sum_{i=0}^n h_i \otimes h_{n-i},\quad \fa n \in \N_+.
\end{equation}
Since the coproduct is an algebra homomorphism, this uniquely determines the coproduct.  The inner product on $\Sy$ is in fact a Hopf pairing of $\Sy$ with itself, in the sense that, for all $a,b,c \in \Sy$, we have
\begin{gather*}
  \langle ab, c \rangle = \langle \nabla(a \otimes b), c \rangle = \langle a \otimes b, \Delta(c) \rangle,\\
  \langle a, bc \rangle = \langle a, \nabla(b \otimes c) \rangle = \langle \Delta(a), b \otimes c \rangle,\\
  \langle a, 1 \rangle = \langle 1, a \rangle = \varepsilon(a).
\end{gather*}
Here we define the inner product on $\Sy \otimes \Sy$ by $\langle a \otimes b, c \otimes d \rangle = \langle a,c \rangle \langle b,d \rangle$.

%
\section[Categorification of the Hopf algebra $\Sy$]{Categorification of the Hopf algebra of symmetric functions} \label{sec:sym-cat}
%

For $n \in \N$, let $S_n$ denote the symmetric group\index{symmetric group} on $n$ letters and let $A_n = \C[S_n]$\index{An@$A_n$} be the corresponding group algebra.  By convention, we set $A_0=A_1=\C$.  By Maschke's Theorem (Lemma~\ref{lem:Maschke}), each $A_n$ is semisimple.  It is well-known (see, for example, \cite[\S7.2, Prop.~1]{Ful97}) that the irreducible representations of $S_n$ are enumerated by the set $\cP(n)$.  To each $\lambda \in \cP(n)$ corresponds the \define{Specht module} $S^\lambda$.  We define
\[
  E_n = S^{(1^n)},\quad L_n = S^{(n)},\quad \fa n \in \N_+.
\]
Then $E_n$ is the sign representation\index{sign representation} of $S_n$ and $L_n$ is the trivial representation\index{trivial representation}.

We have natural inclusions
\[
  S_m \times S_n \hookrightarrow S_{m+n},\quad \fa m,n \in \N,
\]
by letting $S_m$ permute the first $m$ letters and $S_n$ permute the last $n$ letters.  (Here we adopt the convention that $S_0$ and $S_1$ are the trivial groups.)  This gives rise to an injective algebra homomorphism
\[
  A_m \otimes A_n \hookrightarrow A_{m+n},\quad \fa m,n \in \N.
\]
(All tensor products in this section will be over $\Z$ unless otherwise indicated.)  This injection endows $A_{m+n}$ with the structure of a left and right $(A_m \otimes A_n)$-module.

\begin{lem} \label{lem:Sn-tower-projective}
  For all $m,n \in \N$, we have that $A_{m+n}$ is a two-sided projective $(A_m \otimes A_n)$-module.
\end{lem}

\begin{proof}
  This follows immediately from the fact that $A_\ell$ is a semisimple algebra for all $\ell \in \N$.
\end{proof}

\begin{rem}
  The sum
  \[
    A = \bigoplus_{n \in \N} A_n
  \]
  is a \define{tower of algebras}.  We refer the reader to~\cite[\S3.1]{BL09} or~\cite[Def.~3.1]{SY13} for the precise definition of a tower of algebras.  The definitions in these two references are slightly different, but both apply to our situation.
\end{rem}

Define\index{A@$\cA$}\index{GA@$\cG_A$}
\[
  \cA = \bigoplus_{n \in \N} A_n\md,\quad \cG_A = \bigoplus_{n \in \N} G_0(A_n).
\]
Then $\cG_A$ has a basis given by the classes of the Specht modules:
\[
  \cG_A = \bigoplus_{\lambda \in \cP} \Z[S^\lambda].
\]
Since $A_n$ is semsimple, we have $G_0(A_n) = K_0(A_n)$ for all $n \in \N$.  Thus~\eqref{eq:bilinear-form} defines a bilinear form $\langle -,- \rangle \colon \cG_A \otimes \cG_A \to \Z$.

For $r \in \N_+$, define
\[
  \cA^{\otimes r} = \bigoplus_{n_1,\dotsc,n_r \in \N} (A_{n_1} \otimes \dotsc \otimes A_{n_r})\md.
\]
We then have the following functors:
\begin{gather}
  \nabla \colon \cA^{\otimes 2} \to \cA,\quad \nabla|_{(A_m \otimes A_n)\md} = \Ind^{A_{m+n}}_{A_m \otimes A_n}, \nonumber \\
  \Delta \colon \cA \to \cA^{\otimes 2},\quad \Delta|_{A_n\md} = \bigoplus_{k + \ell = n} \Res^{A_n}_{A_k \otimes A_\ell}, \nonumber \\
  \eta \colon \Vect_\C \to \cA,\quad \eta(V) = V \in A_0\md \text{ for } V \in \Vect_\C, \label{eq:Sn-Hopf-functors} \\
  \varepsilon \colon \cA \to \Vect_\C,\quad \varepsilon(V) =
  \begin{cases}
    V & \text{if } V \in A_0\md, \\
    0 & \text{otherwise}.
  \end{cases} \nonumber
\end{gather}
Here we have identified $A_0\md$ with the category $\Vect_\C$\index{Vect@$\Vect$} of finite-dimensional complex vector spaces.  Since
\[
  \Ind_{A_m \otimes A_n}^{A_{m+n}} = A_{m+n} \otimes_{A_m \otimes A_n} - \quad \text{and} \quad \Res^{A_{m+n}}_{A_m \otimes A_n} = {_{A_m \otimes A_n} A_{m+n}} \otimes_{A_{m+n}} -,
\]
it follows from Lemma~\ref{lem:Sn-tower-projective} that the above functors are exact, and so we have induced maps
\begin{equation} \label{eq:Sn-tower-Hopf-structure}
  \nabla \colon \cG_A \otimes \cG_A \to \cG_A,\quad \Delta \colon \cG_A \to \cG_A \otimes \cG_A,\quad \eta \colon \Z \to \cG_A,\quad \varepsilon \colon \cG_A \to \Z,
\end{equation}
where we have used the fact that $\cK_0(\Vect_\C) \cong \Z$ (see Example~\ref{eg:Fmd-split-Groth-group}) and that $G_0(A_m \otimes A_n) \cong G_0(A_m) \otimes G_0(A_n)$ (see Exercise~\ref{prob:Groth-of-product}).

\begin{prop} \label{prop:Sym-Hopf-categorification}
  The maps~\eqref{eq:Sn-tower-Hopf-structure} endow $\cG_A$ with the structure of a graded connected Hopf algebra.  Furthermore, the $\Z$-linear map
  \[
    \varphi_A \colon \cG_A \to \Sy,\quad [S^\lambda] \mapsto s_\lambda,\quad \fa \lambda \in \cP,
  \]
  is an isomorphism of Hopf algebras.  Under this map we have
  \[
    [E_n] \mapsto e_n,\quad [L_n] \mapsto h_n,\quad \fa n \in \N_+.
  \]
  Furthermore, we have
  \[
    \langle a,b \rangle = \langle \varphi_A(a), \varphi_A(b) \rangle,\quad \fa a,b \in \cG_A.
  \]
\end{prop}

\begin{rem}
  Proposition~\ref{prop:Sym-Hopf-categorification} can be summarized as the statement that the categories of modules over group algebras of symmetric groups, together with the functors~\eqref{eq:Sn-Hopf-functors}, give a categorification of $\Sy$ as a Hopf algebra.  This categorification gives a distinguished basis, namely the Schur functions.  This result was known even before the word categorification entered the lexicon.  The interested reader can find the details of the proof that $\varphi_A$ is an isomorphism of algebras, for instance, in \cite[Th.~7.3]{Ful97}.  The fact that it is also an isomorphism of coalgebras then follows from a duality argument.
\end{rem}

While we will not give a full proof of Proposition~\ref{prop:Sym-Hopf-categorification} in these notes, let us at least see the categorification of the equations~\eqref{eq:sym-coproduct}.  Consider the sign representation $E_n$ of $S_n$.  So $E_n$ is the one-dimensional representation of $S_n$ such that each simple transposition acts as $-1$.  It follows that, for $i=0,\dotsc,n$,
\[
  \Res^{A_n}_{A_i \otimes A_{n-i}} (E_n) = {_{A_i \otimes A_{n-i}} A_n} \otimes_{A_n} E_n
\]
is a one-dimensional $(A_i \otimes A_{n-i})$-module, and each simple transposition of $A_i$ and $A_{n-i}$ acts as $-1$.  Thus, $\Res^{A_n}_{A_i \otimes A_{n-i}} \cong E_i \otimes E_{n-i}$.  Therefore,
\[
  \Delta(E_n) = \bigoplus_{i=0}^n \Res^{A_n}_{A_i \otimes A_{n-i}} (E_n) = \bigoplus_{i=0}^n E_i \otimes E_{n-i}.
\]
This is a categorification of the equation $\Delta(e_n) = \sum_{i=0}^n e_i \otimes e_{n-i}$.  A similar argument shows that
\[
  \Delta(L_n) = \bigoplus_{i=0}^n L_i \otimes L_{n-i},
\]
which categorifies the equation $\Delta(h_n) = \sum_{i=0}^n h_i \otimes h_{n-i}$.

\begin{rem}
  The category $\cA$ is, in some ways, simpler than the category $\cN$ used in the categorification of the polynomial representation of the Weyl algebra, in the sense that $\cA$ is a semisimple category, whereas $\cN$ is not.  On the other hand, there is only one simple $N_n$-module for $n \in \N$, which made explicit computations easier.
\end{rem}

\Exercises

\medskip

\begin{prob} \label{prob:Groth-of-product}
  Suppose that $B_1$ and $B_2$ are finite-dimensional unital associative algebras over a field $\F$.  Show that $G_0(B_1 \otimes_\F B_2) \cong G_0(B_1) \otimes_\Z G_0(B_2)$ and $K_0(B_1 \otimes_\F B_2) \cong K_0(B_1) \otimes_\Z K_0(B_2)$ as $\Z$-modules.
\end{prob}

%
\section{The Heisenberg algebra and its Fock space representation}
%

The (infinite-dimensional) \define{Heisenberg algebra} $\fh_\Q$ is usually defined to be the unital associative algebra over $\Q$ (one can also work over $\R$ or $\C$) generated by $\{p_n,q_n\}_{n \in \N_+}$, with defining relations
\begin{equation} \label{eq:fh-usual-presentation}
  p_mp_n=p_np_m,\quad q_m,q_n=q_nq_m,\quad q_mp_n = p_nq_m + n \delta_{m,n},\quad \fa m,n \in \N_+.
\end{equation}
However, it turns out that this presentation is not very well suited to categorification.  The reason is essentially that the generators $p_n$, $n \in \N_+$, will correspond to power sum symmetric functions, and these do not generate $\Sy$ (although they do generate $\Sy \otimes_\Z \Q$ over $\Q$, as noted in Section~\ref{sec:sym}).  Thus, we need to come up with an ``integral version'' of the Heisenberg algebra.

For any $f \in \Sy$, we have the operator $\Sy \to \Sy$ given by multiplication by $f$.  We will denote this operator again by $f$.  The adjoint $f^*$ of the operator $f$ defines a map $\Sy^* \to \Sy^*$.  However, via the inner product on $\Sy$, we may identify $\Sy^*$ with $\Sy$.  Thus, we can view $f^*$ as an operator on $\Sy$.  More explicitly, $f^* \colon \Sy \to \Sy$ is defined by the condition
\[
  \langle a, f^*(b) \rangle = \langle fa, b \rangle, \quad \fa a,b \in \Sy.
\]
Since the inner product is nondegenerate, the above condition uniquely determines $f^*(b)$.

\begin{defin}[Heisenberg algebra and its Fock space representation]
  We define the (integral) \define{Heisenberg algebra} $\fh$ to be the subalgebra of $\End_\Z (\Sy)$ generated by left multiplication by elements of $\Sy$, together with the operators $f^*$ for all $f \in \Sy$.  The algebra $\fh$ acts naturally on $\Sy$.  We call this the \define{Fock space} representation of $\fh$ and denote it by $\cF$.
\end{defin}

The representation $\cF$ has obvious submodules, namely $n \cF$ for $n \in \Z$.  However, we have the following result characterizing $\cF$.

\begin{prop}[Stone--von Neumann Theorem] \label{prop:Stone-von-Neumann}
  The Fock space representation $\cF$ of $\fh$ is faithful and $\cF \otimes_\Z \Q$ is irreducible as an $\fh_\Q$-module. Any $\fh$-module generated by a nonzero element $v$ with $\Z v \cong \Z$ (as $\Z$-modules) and $\Sy^*(v) = 0$ is isomorphic to $\cF$.
\end{prop}

\begin{rem}
  Proposition~\ref{prop:Stone-von-Neumann} is analogous to the characterization of the module $R_\Q$ for the Weyl algebra (see Section~\ref{sec:Weyl-algebra}).  We need the condition $\Z v \cong \Z$ in Proposition~\ref{prop:Stone-von-Neumann} to rule out the possibility $\cF/n\cF$ for some $n > 1$.  If we work over a field instead of over $\Z$, then this condition is not needed.
\end{rem}

We have (see Exercise~\ref{prob:fh-as-Z-module})
\begin{equation} \label{eq:fh-Z-module}
  \fh \cong \Sy \otimes_\Z \Sy^* \quad \text{(as $\Z$-modules)},
\end{equation}
where $\Sy^* = \{f^*\ |\ f \in \Sy\}$ and where the factors $\Sy$ and $\Sy^*$ (more precisely, $\Sy \otimes 1$ and $1 \otimes \Sy^*$) are in fact subalgebras.  Therefore, to find generating sets for $\fh$, it suffices to find generating sets for $\Sy$ (since the adjoints of the elements of such a set will generate $\Sy^*$).  If we work over $\Q$ and choose the generating set $\{p_n\}_{n \in \N_+}$ for $\Sy \otimes_\Z \Q$ and $\{q_n=p_n^*\}_{n \in \N_+}$ for $\Sy^* \otimes_\Z \Q$, then we obtain exactly the presentation~\eqref{eq:fh-usual-presentation}.  \comments{Add this as an exercise, with enough info to solve it (might need the value of the bilinear form on the $p_n$).}  Thus, $\fh_\Q = \fh \otimes_\Z \Q$.  This justifies our calling $\fh$ an integral version of the usual Heisenberg algebra.

Since we want to work over $\Z$, the natural generating sets to use are the elementary or complete symmetric functions, since these generate $\Sy$ over $\Z$.  We will choose generators $\{e_n,h_n^*\}_{n \in \N_+}$.  Since $\{e_n\}_{n \in \Z}$ and $\{h_n\}_{n \in \N_+}$ are free polynomial generating sets for $\Sy$, it follows from~\eqref{eq:fh-Z-module} that we can determine a complete set of relations by computing the commutation relations between the $e_n$ and $h_m^*$.

\begin{lem} \label{lem:he-commutation-relation}
  We have
  \[
    h_m^* e_n = e_n h_m^* + e_{n-1} h_{m-1}^*,\quad \fa m,n \in \N_+.
  \]
  Here we adopt the convention that $e_n =0$ and $h_n=0$ for $n < 0$.
\end{lem}

\begin{proof}
  See~\cite[\S6.2]{SY13}.
\end{proof}

\begin{cor} \label{cor:Heis-alg-integral-presentation}
  The Heisenberg algebra\index{Heisenberg algebra} $\fh$ is the unital associative $\Z$-algebra (i.e.\ unital ring) with generators $\{e_n,h_n^*\}_{n \in \N_+}$ and relations
  \[
    e_m e_n = e_n e_m,\quad h_m^* h_n^* = h_n^* h_m^*,\quad h_m^* e_n = e_n h_m^* + e_{n-1} h_{m-1}^*,\quad \fa n,m \in \N_+.
  \]
\end{cor}

The presentation of $\fh$ given in Corollary~\ref{cor:Heis-alg-integral-presentation} is the one that seems to be best suited to strong categorification of $\fh$.

\Exercises

\medskip

\begin{prob} \label{prob:fh-as-Z-module}
  Prove that $\fh \cong \Sy \otimes_\Z \Sy^*$ as a $\Z$-module.  You may use Lemma~\ref{lem:he-commutation-relation}.  \emph{Hint:} Use the grading on $\Sy$ (and the corresponding grading on $\Sy^*$) to define a natural grading on $\fh$.  Use this grading to show that the multiplication map $\Sy \otimes \Sy^* \to \fh$, $f \otimes g^* \mapsto fg^*$, is an isomorphism of $\Z$-modules.
\end{prob}

%
\section{Weak categorification of Fock space} \label{sec:Fock-weak-cat}
%

Our goal in this section is to categorify the Fock space representation $\cF$ of the Heisenberg algebra $\fh$.  Our underlying category will be the category $\cA$ used in Section~\ref{sec:sym-cat}.  This choice corresponds to the fact that the underlying $\Z$-module of $\cF$ is $\Sy$.  What we must do is to define exact functors on this category that categorify the action of $\fh$ on $\cF$.

Suppose that $\ell,m \in \N$, $L$ is an $(A_\ell \otimes A_m)$-module, and $M \in A_m\md$.  Then $L$ can be considered as an $A_m$-module and $\Hom_{A_m}(M, L)$ is naturally an $A_\ell$-module via the action \[
  (a \cdot f)(m) = (1 \otimes a)f(m),\quad \fa a \in A_\ell,\ f \in \Hom_{A_{m}}(M,L),\ m \in M.
\]
Now, for each $M \in A_m\md$, $m \in \N$, we define two functors $\Ind_M, \Res_M \colon \cA \to \cA$\index{IndM@$\Ind_M$}\index{ResM@$\Res_M$} by
\begin{gather*}
  \Ind_M(N) = \Ind_{A_m \otimes A_n}^{A_{m+n}} (M \otimes N) \in A_{m+n}\md,\quad \fa N \in A_n\md,\ n \in \N,\\
  \Res_M(N) = \Hom_{A_m} (M, \Res^{A_n}_{A_{n-m} \otimes A_m}(N)) \in A_{n-m}\md,\quad \fa N \in A_n\md,\ n \in \N,
\end{gather*}
where $\Res_M(N)$ is interpreted to be the zero object of $\cA$ if $n-m<0$, and we will always interpret $M \otimes N$, for an $A_m$-module $M$ and an $A_n$-module $N$, to be the $A_m \otimes A_n$-module $M \otimes_\C N$ (i.e.\ the outer tensor product module).  These functors are both exact (Exercise~\ref{prob:heisenberg-functors-exact}) and thus induce operators on $\cG_A$.

\begin{prop}[Naive categorification of Fock space] \label{prop:fock-space-weak-cat}
  For all $M \in \cA$, the following diagrams commute:
  \begin{equation} \label{eq:weak-Fock-cat-diagrams}
    \xymatrix{
      \cG_A \ar[rr]^{[\Ind_M]} \ar[d]_{\varphi_A} & & \cG_A \ar[d]^{\varphi_A} \\
      \Sy \ar[rr]^{\varphi_A([M])} & & \Sy
    }
    \qquad \qquad
    \xymatrix{
      \cG_A \ar[rr]^{[\Res_M]} \ar[d]_{\varphi_A} & & \cG_A \ar[d]^{\varphi_A} \\
      \Sy \ar[rr]^{\varphi_A([M])^*} & & \Sy
    }
  \end{equation}
  In other words, the category $\cA$, isomorphism $\varphi_A$, and functors $\Res_M$, $\Ind_M$, $M \in \cA$, give a naive categorification of the Fock space representation of the Heisenberg algebra.
\end{prop}

\begin{proof}
  Let $M \in A_m\md$, $N \in A_n\md$, $n,m \in \N$.  Then we have
  \begin{multline*}
    \varphi_A \circ [\Ind_M]([N]) = \varphi_A([\Ind_M(N)]) = \varphi_A([\Ind_{A_m \otimes A_n}^{A_{m+n}}(M \otimes N)]) \\
    = \varphi_A(\nabla([M] \otimes [N])) = \nabla(\varphi_A([M]) \otimes \varphi_A([N])) = \varphi_A([M]) \varphi_A([N]),
  \end{multline*}
  by Proposition~\ref{prop:Sym-Hopf-categorification}.  Thus the left hand diagram in~\eqref{eq:weak-Fock-cat-diagrams} commutes.

  For all $M \in A_m\md$, $N \in A_n\md$, $L \in A_{n-m}\md$, $n,m \in \N$,
  \begin{align*}
    \langle \varphi_A([L]), \varphi_A \circ [\Res_M]([N]) \rangle
    &= \langle [L], [\Res_M]([N]) \rangle \\
    &= \langle [L], [\Hom_{A_m}(M, \Res^{A_n}_{A_{n-m} \otimes A_m}(N))]\rangle \\
    &= \dim_\C \Hom_{A_{n-m}} (L, \Hom_{A_m}(M, \Res^{A_n}_{A_{n-m} \otimes A_m}(N))) \\
    &= \dim_\C \Hom_{A_{n-m} \otimes A_m} (L \otimes M, \Res^{A_n}_{A_{n-m} \otimes A_m}(N)) \\
    &= \langle [L] \otimes [M], \Delta([N]) \rangle \\
    &= \langle \nabla([L] \otimes [M]), [N] \rangle \\
    &= \langle \varphi_A(\nabla([L] \otimes [M])), \varphi_A([N]) \rangle \\
    &= \langle \varphi_A([L]) \varphi_A([M]), \varphi_A([N]) \rangle \\
    &= \langle \varphi_A([L]), \varphi_A([M])^* (\varphi_A([N])) \rangle.
  \end{align*}
  Thus, $\varphi_A \circ [\Res_M]([N]) = \varphi_A([M])^* \circ \varphi_A([N])$, by the nondegeneracy of the bilinear form.  Hence the right hand diagram in~\eqref{eq:weak-Fock-cat-diagrams} commutes.
\end{proof}

\begin{rem}
  Note that the above naive categorification is independent of the presentation of the Heisenberg algebra.  In this sense, it is even better than the naive categorification of Definition~\ref{def:naive-cat}.
\end{rem}

\begin{lem} \label{lem:Sn-nabla-commutative}
  For all $M \in A_m\md$, $N \in A_n\md$, $m,n \in \N$, we have
  \[
    \nabla(M \otimes N) \cong \nabla(N \otimes M) \quad \text{(as $A_{m+n}$-modules)}.
  \]
\end{lem}

\begin{proof}
  One can either prove this directly (see, for example, \cite[Lem.~4.5]{SY13} for a direct proof in a more general setting) or note that, by Proposition~\ref{prop:fock-space-weak-cat}, we have
  \[
    \varphi_A([\nabla(M \otimes N)]) = \varphi_A([M]) \varphi_A([N]) = \varphi_A([N]) \varphi_A([M]) = \varphi_A([\nabla(N \otimes M)]),
  \]
  where the second equality follows from the fact that $\Sy$ is commutative.  Since $\varphi_A$ is an isomorphism, it follows that $[\nabla(M \otimes N)] = [\nabla(N \otimes M)]$.  Now, since $A_{m+n}$ is semisimple, the images of two modules in the Grothendieck group are equal if and only if the modules themselves are isomorphic (Exercise~\ref{prob:semisimple-module-isom-equal-in-GG}).  Thus $\nabla(M \otimes N) \cong \nabla(N \otimes M)$, as desired.
\end{proof}

\begin{prop}[Weak categorification of Fock space] \label{prop:Fock-weak-cat}
  We have the following isomorphisms of functors for all $m,n \in \N_+$:
  \begin{gather}
    \Ind_{E_m} \circ \Ind_{E_n} \cong \Ind_{E_n} \circ \Ind_{E_m},\quad \Res_{L_m} \circ \Res_{L_n} \cong \Res_{L_n} \circ \Res_{L_m}, \label{eq:Fock-weak-cat-rel1} \\
    \Res_{L_m} \circ \Ind_{E_n} \cong (\Ind_{E_n} \circ \Res_{L_m}) \oplus (\Ind_{E_{n-1}} \circ \Res_{L_{m-1}}). \label{eq:Fock-weak-cat-rel2}
  \end{gather}
  In other words, we have a weak categorification of the Fock space representation $\cF$ of the Heisenberg algebra $\fh$, with the functors $\Ind_{E_n}$, $\Res_{L_n}$, categorifying the action of the generators $e_n$, $h_n^*$, for $n \in \N$.
\end{prop}

\begin{proof}
  In fact, for the isomorphisms~\eqref{eq:Fock-weak-cat-rel1}, we can prove something more general.  Suppose $M \in A_m\md$, $N \in A_n\md$, and $V \in A_v\md$.  Then we have
  \begin{align*}
    \Ind_M \circ \Ind_N (V) &= \Ind^{A_{m+n+v}}_{A_m \otimes A_{n + v}} \left( M \otimes \Ind^{A_{n + v}}_{A_n \otimes A_v} (N \otimes V) \right) \\
    &\cong \Ind^{A_{m+n+v}}_{A_m \otimes A_{n + v}} \Ind^{A_m \otimes A_{n + v}}_{A_m \otimes A_n \otimes A_v} (M \otimes N \otimes V) \\
    &\cong \Ind^{A_{m+n+v}}_{A_m \otimes A_n \otimes A_v} (M \otimes N \otimes V) \\
    &\cong \Ind^{A_{m+n+v}}_{A_{m+n} \otimes A_v} \Ind^{A_{m+n} \otimes A_v}_{A_m \otimes A_n \otimes A_v} (M \otimes N \otimes V) \\
    &\cong \Ind^{A_{m+n+v}}_{A_{m+n} \otimes A_v} \left( \Ind^{A_{m+n}}_{A_m \otimes A_n} (M \otimes N) \otimes V \right) \\
    &\cong \Ind_{\nabla(M \otimes N)} V.
  \end{align*}
  Since each of the above isomorphisms is natural in $V$, this proves that we have an isomorphism of functors $\Ind_M \circ \Ind_N \cong \Ind_{\nabla(M \otimes N)}$.  Therefore, by Lemma~\ref{lem:Sn-nabla-commutative}, for $m, n \in \N$, we have
  \[
    \Ind_{E_m} \circ \Ind_{E_n} \cong \Ind_{\nabla(E_m \otimes E_n)} \cong \Ind_{\nabla(E_n \otimes E_m)} \cong \Ind_{E_n} \circ \Ind_{E_m}.
  \]

  Similarly, we have
  \begin{align*}
    \Res_M \circ \Res_N (V) &= \Hom_{A_m} (M, \Res^{A_{v-n}}_{A_{v-n-m} \otimes A_m} \Hom_{A_n} (N, \Res^{A_v}_{A_{v-n} \otimes A_n} V)) \\
    &\cong \Hom_{A_m} (M, \Hom_{A_n} (N, \Res^{A_v}_{A_{v-m-n} \otimes A_m \otimes A_n} V)) \\
    &\cong \Hom_{A_m \otimes A_n} (M \otimes N, \Res^{A_v}_{A_{v-m-n} \otimes A_m \otimes A_n} V) \\
    &\cong \Hom_{A_m \otimes A_n} (M \otimes N, \Res^{A_{v-m-n} \otimes A_{m+n}}_{A_{v-m-n} \otimes A_m \otimes A_n} \Res^{A_v}_{A_{v-m-n} \otimes A_{m+n}} V) \\
    &\cong \Hom_{A_{m+n}} (\Ind^{A_{m+n}}_{A_m \otimes A_n} (M \otimes N),\Res^{A_v}_{A_{v-m-n} \otimes A_{m+n}} V) \\
    &\cong \Res_{\nabla(M \otimes N)} V,
  \end{align*}
  where, in the second-to-last isomorphism, we have used the fact that restriction is right adjoint to induction.  Since each of the above isomorphisms is natural in $V$, this proves that we have an isomorphism of functors $\Res_M \circ \Res_N \cong \Res_{\nabla(M \otimes N)}$.  Therefore, by Lemma~\ref{lem:Sn-nabla-commutative}, for $m,n \in \N$, we have
  \[
    \Res_{L_m} \circ \Res_{L_n} \cong \Res_{\nabla(L_m \otimes L_n)} \cong \Res_{\nabla(L_n \otimes L_m)} \cong \Res_{L_n} \circ \Res_{L_m}.
  \]

  Finally, for $m,n,v \in \N$ and $V \in A_v\md$, we have
  \begin{align*}
    \Res_{L_m} &\circ \Ind_{E_n} (V)
    = \Hom_{A_m}(L_m, \Res^{A_{n+v}}_{A_{n+v-m} \otimes A_m} \Ind_{A_n \otimes A_v}^{A_{n+v}} (E_n \otimes V)) \\
    &\cong \Hom_{A_m} \left( L_m, \bigoplus_{s+t=m} \Ind_{A_{n-s} \otimes A_s \otimes A_{v-t} \otimes A_t}^{A_{n+v-m} \otimes A_m} \Res^{A_n \otimes A_v}_{A_{n-s} \otimes A_s \otimes A_{v-t} \otimes A_t} (E_n \otimes V) \right) \\
    &\cong \Hom_{A_m} \left( L_m, \bigoplus_{s+t=m} \Ind_{A_{n-s} \otimes A_s \otimes A_{v-t} \otimes A_t}^{A_{n+v-m} \otimes A_m} \left( E_{n-s} \otimes E_s \otimes \Res^{A_v}_{A_{v-t} \otimes A_t} (V) \right) \right) \\
    &\cong \Hom_{A_m} \left( L_m, \bigoplus_{s+t=m} \Ind_{A_{n-s} \otimes A_{v-t} \otimes A_m}^{A_{n+v-m} \otimes A_m} \Ind_{A_{n-s} \otimes A_s \otimes A_{v-t} \otimes A_t}^{A_{n-s} \otimes A_{v-t} \otimes A_m} \left( E_{n-s} \otimes E_s \otimes \Res^{A_v}_{A_{v-t} \otimes A_t} (V) \right) \right) \\
    &\cong \bigoplus_{s+t=m} \Ind_{A_{n-s} \otimes A_{v-t}}^{A_{n+v-m}} \Hom_{A_m} \left( L_m, \Ind_{A_{n-s} \otimes A_s \otimes A_{v-t} \otimes A_t}^{A_{n-s} \otimes A_{v-t} \otimes A_m} \left( E_{n-s} \otimes E_s \otimes \Res^{A_v}_{A_{v-t} \otimes A_t} (V) \right) \right) \\
    &\cong \bigoplus_{s+t=m} \Ind_{A_{n-s} \otimes A_{v-t}}^{A_{n+v-m}} \Hom_{A_s \otimes A_t} \left( \Res^{A_m}_{A_s \otimes A_t} (L_m), \left( E_{n-s} \otimes E_s \otimes \Res^{A_v}_{A_{v-t} \otimes A_t} (V) \right) \right) \\
    &\cong \bigoplus_{s+t=m} \Ind_{A_{n-s} \otimes A_{v-t}}^{A_{n+v-m}} \Hom_{A_s \otimes A_t} \left( L_s \otimes L_t, \left( E_{n-s} \otimes E_s \otimes \Res^{A_v}_{A_{v-t} \otimes A_t} (V) \right) \right),
  \end{align*}
  where, in the first isomorphism, we used the Mackey Theorem relating induction and restriction, and, in the second-to-last isomorphism, we used the fact that restriction is left (and right) adjoint to induction in the current setting (this is the \define{Frobenius Reciprocity Theorem} for finite groups).  Now, $\Hom_{A_s} (L_s, E_s)=0$ unless $s=0,1$, since, in that case, $L_s$ and $E_s$ are both the trivial $\C$-module.  Thus
  \begin{align*}
    \Res_{L_m} \circ \Ind_{E_n} (V)
    &\cong \bigoplus_{s=0,1} \Ind_{A_{n-s} \otimes A_{v-m+s}}^{A_{n+v-m}} \Hom_{A_{m-s}} \left( L_{m-s}, \left( E_{n-s} \otimes \Res^{A_v}_{A_{v-m+s} \otimes A_{m-s}} (V) \right) \right) \\
    &\cong \bigoplus_{s=0,1} \Ind_{A_{n-s} \otimes A_{v-m+s}}^{A_{n+v-m}} \left( E_{n-s} \otimes \left( \Hom_{A_{m-s}} \left( L_{m-s}, \Res^{A_v}_{A_{v-m+s} \otimes A_{m-s}} (V) \right) \right) \right) \\
    &\cong \Ind_{E_n} (\Res_{L_m}(V)) \oplus \Ind_{E_{n-1}} (\Res_{L_{m-1}}(V)).
  \end{align*}
  Since the above isomorphisms are natural in $V$, this proves~\eqref{eq:Fock-weak-cat-rel2}.
\end{proof}

\begin{rem}
  In the proof of the isomorphisms~\eqref{eq:Fock-weak-cat-rel1}, we actually proved the isomorphisms
  \[
    \Ind_M \circ \Ind_N \cong \Ind_{\nabla(M \otimes N)},\quad \Res_M \circ \Res_N \cong \Res_{\nabla(M \otimes N)},\quad \fa M,N \in \cA.
  \]
  It is also possible to prove (see~\cite[Th.~3.18]{SY13}) that
  \[
    \Res_M \circ \Ind_N \cong \nabla \Res_{\Delta(M)} (N \otimes -),\quad \fa M,N \in \cA,
  \]
  which is a generalization of~\eqref{eq:Fock-weak-cat-rel2} (see~\cite[\S6.3]{SY13} for details).  Thus, we actually have a presentation-independent weak categorification of the Fock space representation of the Heisenberg algebra.
\end{rem}

\Exercises

\medskip

\begin{prob} \label{prob:heisenberg-functors-exact}
  Show that the functors $\Ind_M$ and $\Res_M$ are exact for all $M \in A_m\md$, $m \in \N$.
\end{prob}

\begin{prob} \label{prob:semisimple-module-isom-equal-in-GG}
  Suppose $B$ is a semisimple algebra.  Show that two $B$-modules $M$ and $N$ are isomorphic if and only if $[M]=[N]$ in $\cG(B)$.
\end{prob}

%
\section{Towers of algebras and the Heisenberg double}
%

The weak categorifications of the polynomial representation of the Weyl algebra given in Section~\ref{sec:Weyl-weak-cat} and the Fock space representation of the Heisenberg algebra given in Section~\ref{sec:Fock-weak-cat} have a lot in common.  In fact, they are both special cases of a more general construction.

A \define{strong tower of algebras} is a graded algebra
\[
  A = \bigoplus_{n \in \N} A_n,
\]
such that each graded piece $A_n$, $n \in \N$, is a finite-dimensional algebra (with a different multiplication than that of $A$), and satisfying certain other natural conditions (see~\cite[\S3.2]{SY13} for the precise definition).  Then
\[
  \cG = \bigoplus_{n \in \N} G_0(A_n) \quad \text{and} \quad \cK = \bigoplus_{n \in \N} K_0(A_n)
\]
are dual graded connected Hopf algebras under the maps induced by the functors~\eqref{eq:Sn-Hopf-functors}.

Now, given a Hopf algebra $H^+$, let $H^-$ be the dual Hopf algebra.  Then one can define a natural algebra structure on $H^+ \otimes H^-$.  The resulting algebra is the \define{Heisenberg double} of $H^+$.  This algebra acts naturally on $H^+$, and we call this the \define{Fock space} representation. Then, in general, the functors
\[
  \Ind_M, \Res_P,\quad M \in A_m\md,\ P \in A_p\pmd,\ m,p \in \N,
\]
can be defined as in Section~\ref{sec:Fock-weak-cat} and categorify the Fock space representation of the Heisenberg double.

Starting with the tower of nilcoxeter algebras, the corresponding Heisenberg double is the Weyl algebra, and its Fock space representation is the polynomial representation.  So we recover the categorification of Section~\ref{sec:Weyl-weak-cat}.  If we instead start with the tower of group algebras of symmetric groups (or Hecke algebras of type $A$ at a generic parameter), the associated Heisenberg double is the Heisenberg algebra and we recover the categorification of Section~\ref{sec:Fock-weak-cat}.  We can also take the tower of \emph{0-Hecke algebras}\index{0-Hecke algebra}.  In this case, $\cG$ and $\cK$ are the Hopf algebras of \emph{quasisymmetric functions}\index{quasisymmetric function} and \emph{noncommutative symmetric functions}\index{noncommutative symmetric function}, respectively.  The corresponding Heisenberg double is called the \define{quasi-Heisenberg algebra}.  We refer the reader to~\cite{SY13} for further details on the categorification of the Fock space representation of the Heisenberg double.

%
\chapter{Strong categorification} \index{strong categorification}\index{categorification!strong}
\pagestyle{headings}
%

In this chapter, we will introduce the concept of a strong categorification.  We will then present, as an example, the (conjectural) strong categorification of the Heisenberg algebra due to Khovanov (see~\cite{Kho10}).

%
\section{Basic algebraic concepts viewed as categories} \label{sec:alg-concepts-as-cats}
%

In this section we make some simple observations about how certain concepts in algebra can, in fact, be viewed as categories.  If $\cC$ is a category, we will write $\Ob \cC$ for the class of objects of $\cC$.  For $X,Y \in \Ob \cC$, we write $\Mor_\cC(X,Y)$ for the class of morphisms from $X$ to $Y$.  By a common abuse of notation, we will often write $X \in \cC$ to mean $X \in \Ob \cC$.  Recall that a \define{monoidal category}\index{category!monoidal} is a category equipped with a tensor product (satisfying certain natural conditions such as associativity and the existence of an identity object).  We refer the reader to \cite[Ch.~VII]{Mac98} for further details.  For our purposes, it suffices to know that the category of modules over a (fixed) commutative ring is monoidal.  Hence, the categories of abelian groups and vector spaces (over a fixed field) are monoidal.

\begin{eg}[Monoids and groups]
  Monoids are the same as one-object categories.  More precisely, if $\cC$ is a category with only one object $X$, then $\Mor_\cC (X,X)$ is a monoid, with multiplication given by composition.  Similarly, groups are the same as one-object categories where all morphisms are isomorphisms (i.e.\ invertible).
\end{eg}

\begin{defin}[Enriched category]
  Suppose $\cM$ is a (concrete) monoidal category.  Then a category $\cC$ is \emph{enriched}\index{enriched category}\index{category!enriched} over $\cM$ (or, is an $\cM$-\emph{category}) if $\Mor_\cC (X,Y) \in \cM$ for all $X,Y \in \cC$ and composition of morphisms in $\cC$ is a morphism in $\cM$,
  \[
    \Mor_\cC (Y,Z) \otimes \Mor_\cC(X,Y) \to \Mor_\cC(X,Z).
  \]
  We require that this composition is associative (up to isomorphism) and that, for all $X \in \cC$, $\Mor_\cC(X,X)$ contains a unit with respect to the composition.
  \comments{Find good reference for more precise definition.}
\end{defin}

\begin{defin}[Enriched functor]
  If $\cC$ and $\cD$ are categories enriched over a monoidal category $\cM$, then an $\cM$-enriched functor\index{enriched functor}\index{functor!enriched} (or a functor enriched over $\cM$) from $\cC$ to $\cD$ is a usual functor $F \colon \cC \to \cD$ such that, for each $X,Y \in \cC$, the induced map $\Mor_\cC(X,Y) \to \Mor_\cD(F(X),F(Y))$ is a morphism in $\cM$.
\end{defin}

\begin{egs}
  \begin{asparaenum}[(a)]
    \item Ordinary categories are enriched over the category of sets (recall that the tensor product for sets is the cartesian product).

    \item By definition, a category is \emph{preadditive}\index{preadditive category}\index{category!preadditive} if it is enriched over the category of abelian groups.

    \item The category of vector spaces over a field $\F$ is enriched over itself, since the space of linear maps $\Hom_\F (V,W)$ is a vector space for all vector spaces $V$ and $W$, and the composition of linear maps is a bilinear operation.
  \end{asparaenum}
\end{egs}

\begin{eg}[Rings]
  A ring is the same as a preadditive category with one object.  More precisely, if $\cC$ is a category with only one object $X$, then $\Mor_\cC(X,X)$ is a ring.
\end{eg}

\begin{defin}[$R$-linear category]
  If $R$ is a commutative ring, then we say a category is $R$-\emph{linear}\index{R-linear@$R$-linear category}\index{category!$R$-linear} if it is enriched over the category of $R$-modules.
\end{defin}

\begin{eg}[Unital associative algebras] \label{eg:alg-as-cat}
  Suppose $R$ is a commutative ring.  Then a unital associative $R$-algebra is the same as a category with one object, enriched over the category of $R$-modules.
\end{eg}

\begin{eg}[Ring with idempotents] \label{eg:alg-idems-as-cat}
  Suppose $R$ is a commutative ring and $B$ is a unital associative $R$-algebra with a set of orthogonal idempotents $\{e_1,\dotsc,e_n\}$ (i.e.\ $e_i e_j = \delta_{i,j} e_i$ for all $i,j=1,\dotsc,n$) such that $1 = e_1 + \dotsb + e_n$.  Then we have
  \[
    B = \bigoplus_{i,j} {_i B_j}, \quad \text{where } {_i B_j} = e_iBe_j.
  \]
  Then $({_i B_j}) ({_k B_\ell}) \subseteq \delta_{j,k} ({_i B_\ell})$ for all $i,j,k,\ell$.  Is it straightforward to verify that $B$, together with the collection of idempotents $\{e_1,\dotsc,e_n\}$, is equivalent to an $R$-linear category with objects $\{1,\dotsc,n\}$, such that the morphisms from $i \to j$ are ${_j B_i}$.
\end{eg}

\begin{eg}[Lusztig's modified enveloping algebra]\label{eg:modified-env-alg} \index{modified enveloping algebra}
  Suppose $\mathfrak{g}$ is a symmetrizable Kac-Moody algebra and let $U(\mathfrak{g})$ be the enveloping algebra of $\mathfrak{g}$.  By the Poincar\'e--Brikhoff--Witt Theorem, we have an isomorphism (as modules over the ground field/ring $R$) $U(\mathfrak{g}) \cong U(\mathfrak{n}^-) \otimes_R U(\mathfrak{h}) \otimes_R U(\mathfrak{n}^+)$.  Lusztig introduced a \define{modified form} of $U(\mathfrak{g})$ by replacing $U(\mathfrak{h})$ by a system of idempotents:
  \[
    \dot{U}(\mathfrak{g}) = U(\mathfrak{n}^-) \otimes_R \left( \bigoplus_{\lambda \in P} R 1_\lambda \right) \otimes_R U(\mathfrak{n}^+),
  \]
  where $P$ is the weight lattice of $\mathfrak{g}$ and $1_\lambda 1_\mu = \delta_{\lambda,\mu} 1_\lambda$ (i.e.\ the idempotents are orthogonal).  A representation of $\dot{U}(\mathfrak{g})$ is equivalent to a representation of $U(\mathfrak{g})$ with a weight space decomposition (the idempotent $1_\lambda$ acts as projection onto the weight space with weight $\lambda$).  However, $\dot{U}(\mathfrak{g})$ can be naturally viewed as a category whose objects are weights and where the set of morphisms from $\lambda$ to $\mu$ is $1_\mu \dot{U}(\mathfrak{g}) 1_\lambda$ for all $\lambda,\mu \in P$.  This modified form is better suited to categorification.  The above discussion also applies to the \emph{quantized} enveloping algebra.
\end{eg}

Now that we have seen that various algebraic concepts, such as monoids, groups, and algebras, can be viewed as categories, we turn our attention to representations.

\begin{eg}[Group actions on sets]
  Consider a group $\cC$, thought of as a one-object category.  What is a functor $F \colon \cC \to \Set$, where $\Set$ is the category of sets?  The single object $X$ of $\cC$ must be mapped by $F$ to an object $A$ of $\Set$ (i.e.\ a set $A$).  Then, each morphism in $\Mor_\cC (X,X)$ (i.e.\ element of the group) is mapped by $F$ to a set automorphism of $A$.  This mapping respects composition.  Therefore, a functor $\cC \to \Set$ is simply an action of a group on a set.
\end{eg}

\begin{eg}
  If $\cC$ is a monoid, group, or algebra (viewed as a one-object category), then a functor $\cC \to \Vect_\F$ (where $\F$ is a field) is a representation of $\cC$.
\end{eg}

Once we view representations in this way, as functors, then we can consider representations in \emph{any} appropriate category (i.e.\ category with sufficient structure).  For instance, if $\cC$ is a group (viewed as a category), then a functor from $\cC$ to the category of topologicial spaces is an action of a group on a topological space.  It also becomes natural to ask what a natural transformation of functors is in this picture.  We leave it as an exercise (Exercise~\ref{prob:nat-trans-rep-hom}) to show that natural transformations correspond to homomorphisms of representations.

\begin{eg}[Directed graphs]
  Consider the category $\cC$ with two objects, $E$ and $V$, and two morphisms $s,t \colon E \to V$.  (We always assume that we have the identity morphisms for each object.)  Then a functor from $\cC$ to the category of sets is a directed graph.  The images of $E$ and $V$ under the functor correspond to the sets of edges and vertices of the directed graph, respectively.  The images of $s$ and $t$ under the functor correspond to the source and target maps of the directed graph, respectively.
\end{eg}

\Exercises

\medskip

\begin{prob} \label{prob:nat-trans-rep-hom}
  Suppose $\cC$ is a group or algebra, viewed as a one-object category, and $\F$ is a field.  If $F,G \colon \cC \to \Vect_\F$ are functors (i.e.\ representations of $\cC$), show that a natural transformation from $F$ to $G$ corresponds to a homomorphism of representations.
\end{prob}

%
\section{2-categories and their Grothendieck groups} \label{sec:2cat}
%

We saw in Section~\ref{sec:alg-concepts-as-cats} that we can view algebras (or algebras with a given set of idempotents) as categories.  If we want to categorify such a thing, we will need to move one step up on the categorical ladder.  In particular, we need the notion of a 2-category.

\begin{defin}[2-category, 2-functor]
  A \define{2-category} is a category enriched over the category of categories.  In particular, for any two objects $X,Y$ of a 2-category $\fC$, the morphisms $\Mor_\fC(X,Y)$ form a category. The objects of this category $\Mor_\fC(X,Y)$ are called \emph{1-morphisms}\index{1-morphism} of $\fC$ and the morphisms of $\Mor_\fC(X,Y)$ are called \emph{2-morphisms}\index{2-morphism} of $\fC$.   Composition (of 2-morphisms) in the category $\Mor_\fC(X,Y)$ is called \define{vertical composition} and is denoted $\circ_1$.  The composition functor
  \[
    \Mor_\fC(Y,Z) \times \Mor_\fC(X,Y) \to \Mor_\fC(X,Z)
  \]
  is called \define{horizontal composition} and is denoted $\circ_0$.  A \define{2-functor} between 2-categories is a functor enriched over the category of categories.
\end{defin}

\begin{eg}[The 2-category of categories] \index{2-category!of categories}
  The category of (small) categories is in fact a 2-category.  Its objects are (small) categories, its 1-morphisms are functors, and its 2-morphisms are natural transformations.
\end{eg}

\begin{eg}[Monoidal categories] \label{eg:monoidal-cat-as-2-cat}
  A 2-category with one object is a (strict) monoidal category.  More precisely, if $\fC$ is a 2-category with one object $X$, then $\Mor_\fC(X,X)$ is a strict monoidal category.  The vertical composition $\circ_1$ is the composition in the monoidal category, while the horizontal composition $\circ_0$ is the tensor product.
\end{eg}

\begin{eg}[The 2-category of algebras and bimodules] \index{2-category!of bimodules}\index{bimodule 2-category}
  Suppose $R$ is a commutative ring.  Then we have a 2-category of bimodules over $R$-algebras.
  \begin{itemize}
    \item The objects are $R$-algebras.
    \item For any two $R$-algebras $B_1,B_2$, the 1-morphisms from $B_1$ to $B_2$ are $(B_2,B_1)$-bimodules.
    \item For any two $(B_2,B_1)$-bimodules $M$ and $N$, the 2-morphisms from $M$ to $N$ are bimodule homomorphisms.
  \end{itemize}
  The composition of 1-morphisms is given by the tensor product of modules.  In other words, if $M$ is a $(B_2,B_1)$-bimodule and $N$ is a $(B_3,B_2)$-bimodule, then their composition is the $(B_3,B_1)$-bimodule $N \otimes_{B_2} M$.
\end{eg}

There is a 2-functor from the 2-category of bimodules over $R$-algebras to the 2-category of categories that
\begin{itemize}
  \item sends an $R$-algebra $B$ to the category $B\Md$ of all $B$-modules,
  \item sends a $(B_2,B_1)$-bimodule $M$ to the tensor product functor
    \[
      (M \otimes_{B_1} -) \colon B_1\Md \to B_2\Md,
    \]
  \item sends a bimodule map to the corresponding natural transformation of functors (see Exercise~\ref{prob:bimodule-hom-natural-trans}).
\end{itemize}
The Eilenberg--Watts Theorem\index{Eilenberg--Watts Theorem} states that the image of this 2-functor consists of the colimit-preserving functors between categories of modules.

%
\section{Strong categorification}
%

To take the Grothendieck group of a 2-category, we take the Grothendieck groups of the morphism categories.

\begin{defin}[Additive 2-category, abelian 2-category, $R$-linear 2-category]
  Suppose $R$ is a commutative ring.  A 2-category is said to be \emph{additive}\index{additive 2-category}\index{2-category!additive}, \emph{abelian}\index{abelian 2-category}\index{2-category!abelian}, or \emph{$R$-linear}\index{R-linear@$R$-linear 2-category}\index{2-category!R-linear@$R$-linear} if it is enriched over the category of additive categories, abelian categories, or $R$-linear categories, respectively.
\end{defin}

\begin{defin}[Grothendieck group of a 2-category]
  The \emph{Grothendieck group}\index{Grothendieck group!of a 2-category} (resp.\ \emph{split Grothendieck group}\index{split Grothendieck group!of a 2-category}) of an abelian (resp.\ additive) 2-category $\fC$ is the category $\cK_0(\fC)$ (resp.\ $\cK_0^\spl(\fC)$) whose objects are the objects of $\fC$ and such that, for all $X,Y \in \Ob \fC$, $\Mor_{\cK_0(\fC)}(X,Y)$ (resp.\ $\Mor_{\cK_0^\spl(\fC)}(X,Y)$) is equal to $\cK_0(\Mor_\fC(X,Y))$ (resp.\ $\cK_0^\spl(\Mor_\fC(X,Y))$), the Grothendieck group (resp.\ split Grothendieck group) of the category $\Mor_\fC(X,Y)$.  Note that $\cK_0(\fC)$ and $\cK_0^\spl(\fC)$ are both preadditive categories (i.e.\ are enriched over the category of abelian groups).  The composition in $\cK_0(\fC)$ or $\cK_0^\spl(\fC)$ is defined by
  \[
    [f] \circ [g] = [f \circ_0 g],\quad \fa f \in \Mor_\fC(Y,Z),\ g \in \Mor_\fC(X,Y),\ X,Y,Z \in \Ob \fC.
  \]
\end{defin}

We are now in a position to define the notion of strong categorification (sometimes simply called categorification).

\begin{defin}[Strong categorification] \label{def:strong-cat}
  Suppose $R$ is a commutative ring and let $\cC$ be an $R$-linear category.  A \define{strong categorification}\index{categorification!strong} of $\cC$ is a pair $(\fC, \varphi)$, where either
  \begin{enumerate}[(a)]
    \item $\fC$ is an additive 2-category and $\varphi \colon \mathcal{K}_0^\spl(\fC) \otimes_\Z R \to \cC$ is an isomorphism, or
    \item $\fC$ is an abelian 2-category and $\varphi \colon \mathcal{K}_0(\fC) \otimes_\Z R \to \cC$ is an isomorphism.
  \end{enumerate}
  Here the tensor product $\otimes_\Z R$ means that we tensor the morphism sets with $R$ over $\Z$, turning the additive categories $\cK_0^\spl(\cC)$ and $\cK_0(\cC)$ into $R$-linear categories.
\end{defin}

\begin{rem} \label{rem:assoc-alg-cat}
  Since a unital associative $R$-algebra can be viewed as a one-object category (see Example~\ref{eg:alg-as-cat}), Definition~\ref{def:strong-cat} includes the definition of the strong categorification of such algebras.  In fact, a 2-category with one object is the same as a monoidal category (see Example~\ref{eg:monoidal-cat-as-2-cat}).
  \comments{Give reference for periodic table.}
  Thus, one can also give the definition of a strong categorification of a unital associative $R$-algebra in terms of monodial categories.  However, one often wants to categorify an algebra together with some collection of idempotents, which can be viewed as a category with multiple objects as in Example~\ref{eg:alg-idems-as-cat}.
\end{rem}

For our first example of strong categorification, we return to the setting of Example~\ref{eg:Maz}.

\begin{eg}[{\cite[Ex.~2.14]{Maz12}}] \label{eg:Maz-strong}
  Let $B=\C[b]/(b^2-2b)$ and let $D = \C[x]/(x^2)$ be the algebra of dual numbers.  Consider the $(D,D)$-bimodule $X = D \otimes_\C D$.  Let $\fC$ be the 2-category with one object $I = D\md$ and with $\Mor_\fC(I,I)$ equal to the full additive subcategory of the category of endofunctors of $I$ consisting of all functors isomorphic to direct sums of copies of $\id = \id_I$ and $F = X \otimes_D -$.  One can check that $F \circ F \cong F \oplus F$ (Exercise~\ref{prob:Maz-functor-isom}).  The classes $[\id]$ and $[F]$ form a basis of $\cK_0^\spl(\Mor_\fC(I,I))$ and the map
  \[
    \varphi \colon \cK_0^\spl(\fC) \otimes_\Z \C \to B,\quad \varphi([\id])=1,\quad \varphi([F]) = b
  \]
  is an isomorphism.  Hence $(\fC,\varphi)$ is a strong categorification of $B$.
\end{eg}

\Exercises

\medskip

\begin{prob} \label{prob:Maz-functor-isom}
  With notation as in Example~\ref{eg:Maz-strong}, show that one has an isomorphism of functors $F \circ F \cong \id \oplus \id$.
\end{prob}

%
\section{The graphical Heisenberg category} \label{sec:Heis-cat}
%

In this section, we define the monoidal category that will yield a categorification of the Heisenberg algebra $\fh$.  This monoidal category was defined by Khovanov in~\cite{Kho10}.  It was $q$-deformed in~\cite{LS13} by replacing group algebras of symmetric groups by Hecke algebras of type $A$.  We also refer the reader to the expository paper~\cite{LS12} for an overview of Heisenberg categorification.  In fact, the presentation below closely follows~\cite[\S3.1]{LS12}.

As noted in Example~\ref{eg:alg-as-cat}, the unital associative $\Z$-algebra $\fh$ can be thought of as a category with one object.  Thus, its categorification should involve a 2-category with one object.  But this is nothing more than a (strict) monoidal category (see Example~\ref{eg:monoidal-cat-as-2-cat}).

We define a $\C$-linear strict monoidal category $\cH'$\index{H:$\cH'$} as follows.  The set of objects is generated by two objects $Q_\up$ and $Q_\down$.  In other words, an arbitrary object of $\cH'$ is a finite direct sum of tensor products $Q_\varepsilon := Q_{\varepsilon_1} \otimes \dots \otimes Q_{\varepsilon_n}$, where $\varepsilon = \varepsilon_1 \dots \varepsilon_n$ is a finite sequence of $\up$ and $\down$ symbols.  The unit object is $\mathbf{1}=Q_\emptyset$.

The space of morphisms $\Mor_{\cH'}(Q_\varepsilon, Q_{\varepsilon'})$ is the $\C$-module generated by  planar diagrams modulo local relations.   The diagrams are oriented compact one-manifolds immersed in the strip $\R \times [0,1]$, modulo rel boundary isotopies.  The endpoints of the one-manifold are located at $\{1,\dots,m\} \times \{0\}$ and $\{1,\dots,k\} \times \{1\}$, where $m$ and $k$ are the lengths of the sequences $\varepsilon$ and $\varepsilon'$ respectively.  The orientation of the one-manifold at the endpoints must agree with the arrows in the sequences $\varepsilon$ and $\varepsilon'$ and triple intersections are not allowed.  For example, the diagram
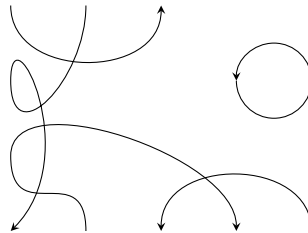
\[
  \begin{tikzpicture}[>=stealth]
    \draw[->] (0,3) .. controls (0,2) and (2,2) .. (2,3);
    \draw[->] (1,3) .. controls (1,2) and (0,1) .. (0,2) .. controls (0,3) and (1,1) .. (0,0);
    \draw[->] (1,0) .. controls (1,1) and (0,0) .. (0,1) .. controls (0,2) and (3,1) .. (3,0);
    \draw[->] (4,0) .. controls (4,1) and (2,1) .. (2,0);
    \draw[->] (3,2) arc(-180:180:.5);
  \end{tikzpicture}
\]
is a morphism from $Q_{\down \up \down \down \up}$ to $Q_{\down \down \up}$ (note that, in this sense, diagrams are read from bottom to top).  An arbitrary morphism is a $\C$-linear composition of such diagrams.  Composition of morphisms is given by the natural vertical gluing of diagrams and the tensor product of morphisms is horizontal juxtaposition.  An endomorphism of $\mathbf{1}$ is a diagram without endpoints.  The local relations are as follows.

\begin{equation} \label{eq:local-relation-basic-symmetric}
\begin{tikzpicture}[>=stealth,baseline=25pt]
  \draw (0,0) .. controls (1,1) .. (0,2)[->];
  \draw (1,0) .. controls (0,1) .. (1,2)[->] ;
  \draw (1.5,1) node {=};
  \draw (2.5,0) --(2.5,2)[->];
  \draw (3.5,0) -- (3.5,2)[->];
\end{tikzpicture}
\end{equation}

\begin{equation} \label{eq:local-relation-braid}
\begin{tikzpicture}[>=stealth,baseline=25pt]
  \draw (0,0) -- (2,2)[->];
  \draw (2,0) -- (0,2)[->];
  \draw (1,0) .. controls (0,1) .. (1,2)[->];
  \draw (2.5,1) node {=};
  \draw (3,0) -- (5,2)[->];
  \draw (5,0) -- (3,2)[->];
  \draw (4,0) .. controls (5,1) .. (4,2)[->];
\end{tikzpicture}
\end{equation}

\begin{equation} \label{eq:local-relation-up-down-double-crossing}
\begin{tikzpicture}[>=stealth,baseline=25pt]
  \draw (0,0) .. controls (1,1) .. (0,2)[<-];
  \draw (1,0) .. controls (0,1) .. (1,2)[->] ;
  \draw (1.5,1) node {=};
  \draw (2.5,0) --(2.5,2)[<-];
  \draw (3.5,0) -- (3.5,2)[->];
  \draw (4,1) node {$-$};
  \draw (4.5,1.75) arc (180:360:.5) ;
  \draw (4.5,2) -- (4.5,1.75) ;
  \draw (5.5,2) -- (5.5,1.75) [<-];
  \draw (5.5,.25) arc (0:180:.5) ;
  \draw (5.5,0) -- (5.5,.25) ;
  \draw (4.5,0) -- (4.5,.25) [<-];
\end{tikzpicture} \qquad \qquad
\begin{tikzpicture}[>=stealth,baseline=25pt]
  \draw (0,0) .. controls (1,1) .. (0,2)[->];
  \draw (1,0) .. controls (0,1) .. (1,2)[<-] ;
  \draw (1.5,1) node {$=$};
  \draw (2.3,0) --(2.3,2)[->];
  \draw (3.3,0) -- (3.3,2)[<-];
\end{tikzpicture}
\end{equation}

\begin{equation} \label{eq:cc-circle-and-left-curl}
\begin{tikzpicture}[>=stealth,baseline=0pt]
  \draw [<-](0,0) arc(180:0:.5);
  \draw (0,0) arc(180:360:.5);
  \draw (1.5,0) node {$=$};
  \draw (2,0) node {$\id$};
  \draw (0,-1) node {};
  \draw (0,1) node {};
\end{tikzpicture} \qquad \qquad
\begin{tikzpicture}[>=stealth,baseline=0pt]
  \draw (-1,0) .. controls (-1,.5) and (-.3,.5) .. (-.1,0) ;
  \draw (-1,0) .. controls (-1,-.5) and (-.3,-.5) .. (-.1,0) ;
  \draw (0,-1) .. controls (0,-.5) .. (-.1,0) ;
  \draw (-.1,0) .. controls (0,.5) .. (0,1) [->] ;
  \draw (0.7,0) node {$=0$};
\end{tikzpicture}
\end{equation}

By \define{local relation}, we mean that any time we see one of the diagrams on the left hand sides of the equations above as a sub-diagram of a larger diagram, we may replace this sub-diagram by the corresponding linear combination of diagrams on the right hand side of the equation.

We already saw in Section~\ref{sec:Fock-weak-cat} that the tower of symmetric groups can be used to categorify the Fock space representation of the Heisenberg algebra, which is a faithful representation.  Thus, it is the behaviour of the symmetric groups that motivates our definition of the local relations above.  Note that relations~\eqref{eq:local-relation-basic-symmetric} and~\eqref{eq:local-relation-braid} are simply the relations defining the symmetric group, if we think of a simple transposition as a crossing of neighbouring upward pointing strands.  (The relation that distant simple transpositions commute comes for free in this graphical description since it follows from the fact that we consider diagrams up to isotopy preserving the boundary.) Relations~\eqref{eq:local-relation-up-down-double-crossing} and~\eqref{eq:cc-circle-and-left-curl} come from the behaviour of induction and restriction between symmetric groups, as we will see.

It turns out that the category $\cH'$ is not quite large enough to yield a categorification of $\fh$.  This is essentially for the following reason.  The upward pointing strands in our diagrams will correspond to the induction $\Ind_{A_n}^{A_{n+1}}$, where $A_n = \C[S_n]$ is the group algebra of the symmetric group.  It follows that a sequence of $m$ upward pointing strands will correspond to the induction $\Ind_{A_n}^{A_{n+m}}$.  One can check (Exercise~\ref{prob:bimodule-induction-correspondence}) that
\[
  \Ind_{A_n}^{A_{n+m}} = (\Ind_{A_m})|_{A_n\md}.
\]
However, our weak categorification of the (faithful) Fock space representation involved the functors $\Ind_M$ for \emph{all} $A_m$-modules $M$.  These functors (more precisely, the corresponding diagrams) are missing from our graphical category.  In attempting to remedy this situation, the key observation is that any simple $A_m$-module $M$ is a direct summand of $A_m$.  Thus, we would like to add missing direct summands to our graphical category $\cH$.

\begin{defin}[Idempotent completion] \label{def:idem-completion}
  Let $\mathcal{C}$ be a category.  The \define{idempotent completion} (or \define{Karoubi envelope}) of $\mathcal{C}$ is the category whose objects are pairs $(X,e)$ where $X$ is an object of $\mathcal{C}$ and $e \in \Mor_\mathcal{C}(X,X)$ is an idempotent endomorphism of $X$ (i.e.\ $e^2=e$).  Morphisms $(X,e) \to (X',e')$ are morphisms $f \colon X \to X'$ in $\mathcal{C}$ such that the diagram
  \[
    \xymatrix{
      X \ar[r]^f \ar[dr]^f \ar[d]_e & X' \ar[d]^{e'} \\
      X \ar[r]_f & X'
    }
  \]
  commutes.  Composition in the idempotent completion is as in $\mathcal{C}$, except that the identity morphism of $(X,e)$ is $e$.
\end{defin}

Intuitively, one should think of the idempotent completion as follows.  If an object $X$ in a category decomposes as a direct sum $X \cong Y \oplus Z$, then the composition
\[
  X \twoheadrightarrow Y \hookrightarrow X,
\]
where the first map is projection onto the summand $Y$ and the second is the inclusion, is an idempotent endomorphism of $X$.  If a category has an idempotent morphism $e$ of an object $X$, we would like this to always correspond to a map as above.  If it does not, we say the category is not \define{idempotent complete}.  In some sense, such a category is ``missing'' a subobject of $X$.  The idempotent completion is a way of formally adding in all such summands -- the object $(X,e$) in the idempotent completion plays the role of the missing summand.  If a category is already idempotent complete (i.e.\ every idempotent morphism corresponds to a projection onto a summand as above), then this category is isomorphic to its idempotent completion.

We are now ready to define the monoidal category that will yield our (conjectural) categorification of the Heisenberg algebra.

\begin{defin}[The Heisenberg category]
  We define the \define{Heisenberg category} $\cH$\index{H@$\cH$} to be the idempotent completion of $\cH'$.
\end{defin}

\Exercises

\medskip

\begin{prob} \label{prob:bimodule-induction-correspondence}
  We use the notation of Section~\ref{ch:heisenberg}.  Show that, for all $m,n \in \N$, we have an isomorphism of functors $\Ind_{A_n}^{A_{m+n}} \cong \Ind_{A_m}|_{A_n\md}$.
\end{prob}

%
\section{Strong categorification of the Heisenberg algebra}
%

It follows from the local relations \eqref{eq:local-relation-basic-symmetric} and \eqref{eq:local-relation-braid} that upward oriented crossings satisfy the relations of $A_n$ and so we have a canonical homomorphism
\begin{equation} \label{eq:An-to-Q+}
  A_n \to \End_{\cH'} (Q_{\up^n}).
\end{equation}
Similarly, since each space of morphisms in $\cH'$ consists of diagrams up to isotopy, downward oriented crossings also satisfy the Hecke algebra relations and give us a canonical homomorphism
\begin{equation} \label{eq:An-to-Q-}
  A_n \to \End_{\cH'} (Q_{\down^n}).
\end{equation}
Introduce the complete symmetrizer and antisymmetrizer
\[
  e(n) = \frac{1}{n!} \sum_{\sigma \in S_n} \sigma,\quad e'(n) = \frac{1}{n!} \sum_{\sigma \in S_n} (-1)^{\ell(\sigma)} \sigma,
\]
where $\ell(\sigma)$ is the length of the permutation $\sigma$.  Both $e(n)$ and $e'(n)$ are idempotents in $A_n$.  We will use the notation $e(n)$ and $e'(n)$ to also denote the image of these idempotents in $\End_{\cH'} (Q_{\up^n})$ and $\End_{\cH'} (Q_{\down^n})$ under the canonical homomorphisms~\eqref{eq:An-to-Q+} and~\eqref{eq:An-to-Q-}.  We then define the following objects in $\cH$:
\[
  S_\down^n = (Q_{\down^n}, e(n)),\quad \Lambda_\up^n = (Q_{\up^n}, e'(n)).
\]

\begin{theorem}[{\cite[Th.~1]{Kho10}}] \label{thm:heisenberg-cat}
  In the category $\cH$, we have
  \begin{gather*}
    S_\down^n \otimes S_\down^m \cong S_\down^m \otimes S_\down^n, \quad
    \Lambda_\up^n \otimes \Lambda_\up^m \cong \Lambda_\up^m \otimes \Lambda_\up^n, \\
    S_\down^n \otimes \Lambda_\up^m \cong \left( \Lambda_\up^m \otimes S_\down^n \right) \oplus \left( \Lambda_\up^{m-1} \otimes S_\down^{n-1} \right).
  \end{gather*}
  We thus have a well-defined $\Z$-algebra homomorphism $\varphi \colon \fh \to \mathcal{K}^\spl_0(\cH)$ given by
  \[
    \varphi(h_n^*) = [S_\down^n],\quad \varphi(e_n) = [\Lambda_\up^n],\quad \fa n \in \N_+.
  \]
  This homomorphism is injective.
\end{theorem}

\begin{rem}
  Note that the map $\varphi$ in Theorem~\ref{thm:heisenberg-cat} is a homomorphism of $\Z$-algebras (as opposed to a functor as in Definition~\ref{def:strong-cat}) since we are viewing one-object additive categories as $\Z$-algebras (see Example~\ref{eg:alg-as-cat}).  It is conjectured in~\cite[Conj.~1]{Kho10} that the map $\varphi$ is, in fact, an isomorphism.  If this is true, then $(\cH,\varphi^{-1})$ is a strong categorification of the Heisenberg algebra.  The difficulty in proving this conjecture lies in our passage from the graphical category $\cH'$ to its idempotent completion $\cH$.  It is difficult to prove that one has found all of the objects in this idempotent completion (equivalently, all of the idempotents in the morphism spaces).  In other, somewhat similar, categorifications, one has a grading on the morphism spaces that allows one to find all of the idempotents (since idempotents must lie in degree zero).  We refer the reader to the expository article~\cite[\S4]{LS12} for further details.
\end{rem}

Injectivity in Theorem~\ref{thm:heisenberg-cat} is proved by relating the graphical category $\cH$ to the weak categorification of Fock space (Proposition~\ref{prop:Fock-weak-cat}).  For $n,m \in \N$, let $_m\bimod_n$ be the category of finite-dimensional $(A_m,A_n)$-bimodules.  One then defines a functor
\[
  \cH \to \bigoplus_{m,n \in \N} {_m\bimod_n}.
\]
Intuitively, this functor takes upwards pointing arrows to the bimodules defining induction and downwards pointing arrows to the bimodules defining restriction.

Now, we know that an $(A_m,A_n)$-bimodule defines a functor $A_m\md \to A_n\md$ by tensoring on the left (see Section~\ref{sec:2cat}).  Thus, if $\Fun(\cC,\cD)$ denotes the category of exact functors from an abelian category $\cC$ to an abelian category $\cD$, we have a functor
\[
  {_m\bimod_n} \to \Fun(A_n\md,A_m\md).
\]
We then define the composition of functors
\[
  F \colon \cH \to \bigoplus_{m,n \in \N} {_m\bimod_n} \to \bigoplus_{m,n \in \N} \Fun(A_n\md,A_m\md) \to \Fun \left( \bigoplus_{n \in \N} A_n\md, \bigoplus_{n \in \N} A_n\md \right).
\]
Thus, $F$ is a representation of $\cH$ on the category $\cA = \bigoplus_{n \in \N} A_n\md$\index{A@$\cA$}.

We then have a commutative diagram
\[
  \xymatrix{
    \cH \ar[r]^(0.35){F} \ar[d]_{\cK_0} & \Fun (\cA, \cA) \ar[d]^{\cK_0} \\
    \cK_0(\cH) \ar[r]^{[F]} & \End_\Z \cF \\
    \fh \ar@{^{(}->}[u]^{\varphi} \ar[ur]
  }
\]
The functors $\Ind_{E_n}$ and $\Res_{L_n}$ that appeared in the weak categorification of Fock space (Proposition~\ref{prop:Fock-weak-cat}) are the images under $F$ of the objects $\Lambda^n_\up$ and $S^n_\down$, respectively.

%
\section{Further directions}
%

While we have seen the definition of (naive, weak, and strong) categorification in these notes, we have only touched on a small subset of the important examples of categorification.  In this final section, we mention a few of the many other examples.

As mentioned in Example~\ref{eg:modified-env-alg}, one can define a modified quantized enveloping algebra that is an algebra with a collection of idempotents.  These idempotents are indexed by the elements of the weight lattice.  For each weight, Lusztig defined a \define{quiver variety} and considered a certain category of perverse sheaves on these varieties.  He then defined a convolution product on these sheaves.  Passage to the Grothendiecck group then recovers the modified quantized enveloping algebra.  The so-called \define{canonical basis} appears naturally from this construction as the classes of simple objects.

Khovanov--Lauda and Rouquier have given categorifications of Kac--Moody algebras (and their quantum analogues).  The definition of these \define{2-Kac-Moody algebras} (or \define{2-quantized enveloping algebras}) by Khovanov--Lauda is diagrammatic, similar to the description of the graphcial Heisenberg category $\cH$ given in Section~\ref{sec:Heis-cat}, while the description by Rouquier is more algebraic.  One can then consider 2-representations of quantum groups, which are 2-functors into other 2-categories (e.g.\ 2-categories of bimodules).

There are important applications of categorification to knot and surface invariants.  The category $k$-Cob of $k$-cobordisms is the category whose objects are oriented $k$-manifolds, and whose morphisms are $(k+1)$-manifolds with boundaries corresponding to their domain and codomain (i.e.\ cobordisms).  Then a $(k+1)$-dimensional topological quantum field theory (TQFT) is a monoidal functor from $k$-Cob to the category of $R$-modules for some ring $R$.  The Reshetikhin--Turaev invariant is a $(0+1)$-dimensional TQFT.  For a fixed simple Lie algebra $\mathfrak{g}$ and a representation $V$ of $\mathfrak{g}$, it is a functor
\[
  0\text{-Cob} \to U_q(\mathfrak{g})\md,
\]
where $U_q(\mathfrak{g})$ is the quantized enveloping algebra associated to $\mathfrak{g}$.  This functor sends the empty 0-manifold to the trivial $U_q(\mathfrak{g})$-module $\Z[q,q^{-1}]$.  Thus, it sends a knot (which is a cobordism from the trivial, i.e.\ empty, 0-manifold to itself) to an endomorphism of the trivial module $\Z[q,q^{-1}]$.  Such an endomorphism is simply multiplication by an element $p \in \Z[q,q^{-1}]$.  If $\mathfrak{g}=\mathfrak{sl}_2$ and $V$ is the standard two-dimensional module, then $p$ is the Jones polynomial of the knot.

\define{Khovanov homology} is a categorification of the Reshetikhin--Turaev invariant.  It is a functor from the category $0$-Cob to the category of categories.  Thus, it maps a 0-manifold to a category.  The empty 0-manifold gets mapped to the category of complexes of graded vector spaces.  It maps 1-cobordisms (i.e.\ tangles) to functors.  Passing to Grothendieck groups recovers the Reshetikhin--Turaev invariant.

Going even further, one can consider extended TQFTs.  Consider the 2-category whose objects are 0-manifolds, whose 1-morphisms are tangles, and whose 2-morphisms are cobordisms between tangles.  One would then like to construct functors from this 2-category to the 2-category of representations of a 2-quantized enveloping algebra.  Doing so should result in richer knot invariants.  One assigns to each knot a homology theory instead of a polynomial.  Since a surface is a cobordism from the trivial tangle to itself, one should obtain polynomial invariants of surfaces.  This is an active area of research.


\printindex

\printbibliography

\end{document}